\documentclass[11pt,a4paper]{amsart}		

\usepackage[utf8]{inputenc}					

\usepackage{lmodern}						
\usepackage[english]{babel}					

\usepackage[normalem]{ulem} 

\usepackage[all,cmtip]{xy}
\usepackage{amsmath,amsfonts,amssymb,amsthm}

\usepackage{todonotes}
\usepackage{tikz-cd}

\usepackage[left=3cm,right=3cm,top=3cm,bottom=3cm]{geometry}

\usepackage{mathtools}						

\DeclareMathAlphabet{\mathbbold}{U}{bbold}{m}{n}	

\usepackage{todonotes}

\theoremstyle{plain}
\newtheorem{thm-intro}{Theorem}
\newtheorem{cor-intro}[thm-intro]{Corollary}

\newtheorem*{theorem*}{Theorem}
\newtheorem{theorem}{Theorem}[section]
\newtheorem{prop}[theorem]{Proposition}

\newtheorem{lemma}[theorem]{Lemma}
\theoremstyle{definition}
\newtheorem{definition}[theorem]{Definition}
\theoremstyle{remark}
\newtheorem{rmk}{Remark}[section]
\newtheorem{example}{Example}[section]

\DeclareMathOperator{\spn}{\mathrm{span}}
\DeclareMathOperator{\rank}{\mathrm{rank}}

\DeclareMathOperator{\supp}{supp}

\DeclareMathOperator{\re}{Re}
\DeclareMathOperator{\im}{Im}
\DeclareMathOperator{\dive}{\mathrm{div}}

\newcommand{\N}{\mathbb{N}}					
\newcommand{\R}{\mathbb{R}}					
\newcommand{\distr}{\mathcal{D}}			
\renewcommand{\epsilon}{\varepsilon}		

\newcommand{\op}{H}
\newcommand{\Veff}{V_{\mathrm{eff}}}
\newcommand{\loc}{\mathrm{loc}}
\newcommand{\comp}{\mathrm{comp}}
\DeclareMathOperator{\dom}{Dom}

\usepackage{hyperref}
\usepackage{xcolor}
\hypersetup{
    colorlinks,
    linkcolor={red!50!black},
    citecolor={blue!50!black},
    urlcolor={blue!80!black}
}

\title{On the essential self-adjointness of singular sub-Laplacians}

\author[Valentina Franceschi]{Valentina Franceschi$^\dagger$}
\address{$^\dagger$FMJH \& IMO, B\^atiment 307, Facult\'e des Sciences d'Orsay, Universit\'e Paris Sud, Orsay}
\email{\href{mailto:valentina.franceschi@math.u-psud.fr}{valentina.franceschi@math.u-psud.fr}}

\author[Dario Prandi]{Dario Prandi$^\flat$}
\address{$^\flat$ CNRS, Laboratoire des Signaux \& Syst\'emes, CentraleSup\'elec, Gif-sur-Yvette, France}
\email{\href{mailto:dario.prandi@l2s.centralesupelec.fr}{dario.prandi@l2s.centralesupelec.fr}}

\author[Luca Rizzi]{Luca Rizzi$^\sharp$}
\address{$^\sharp$ Univ. Grenoble Alpes, CNRS, Institut Fourier, F-38000 Grenoble, France}
\email{\href{mailto:luca.rizzi@univ-grenoble-alpes.fr}{luca.rizzi@univ-grenoble-alpes.fr}}

 
\begin{document}

\begin{abstract}
We prove a general essential self-adjointness criterion for sub-Laplacians on complete sub-Riemannian manifolds, defined with respect to singular measures. We also show that, in the compact case, this criterion implies discreteness of the sub-Laplacian spectrum even though the total volume of the manifold is infinite.

As a consequence of our result, the intrinsic sub-Laplacian (i.e.\ defined w.r.t.\ Popp's measure) is essentially self-adjoint on the equiregular connected components of a sub-Riemannian manifold. This settles a conjecture formulated by Boscain and Laurent (Ann.\ Inst.\ Fourier, 2013), under mild regularity assumptions on the singular region, and when the latter does not contain characteristic points.
\end{abstract}

\maketitle

\section{Introduction}
\label{sec:intro}
It is well known that geometric singularities of a Riemannian structure can act as barriers for heat diffusion, wave propagation, and the evolution of quantum particles. Most surprisingly, this occurs even when the underlying Riemannian structure is not complete, and classical particles, whose trajectories are described by geodesics, can escape from the manifold in finite time. One of the simplest cases where this behavior can be observed is the \emph{Grushin structure} given by the singular metric
\begin{equation}
g = dx\otimes dx + \frac{1}{x^2} dy \otimes dy.
\end{equation}
This Riemannian structure on $\R^2 \setminus \{x=0\}$ is clearly not geodesically complete, as almost all geodesics cross the singular region $\mathcal{Z} = \{x=0\}$ in finite time. Moreover, the associated Riemannian volume $\frac{1}{|x|} dx\wedge dy$ explodes on $\mathcal{Z}$ and hence the corresponding Laplace-Beltrami operator presents both a degeneration and a singular drift on $\mathcal Z$:
\begin{equation}
\Delta = \partial_x^2 + x^2 \partial_y^2 - \frac{1}{x} \partial_x.
\end{equation}
It is not hard to show that $\Delta$ with domain $\dom(\Delta)=C^\infty_c(M)$ is essentially self-adjoint on $L^2(M)$, where $M$ is either $\R^2\setminus\mathcal{Z}$ or one of its two connected components. 
As a consequence, by Stone Theorem, there exists a unique unitary Schr\"odinger evolution defined for any initial datum in $L^2(M)$, without the need to impose boundary conditions.
From a physical viewpoint this means that quantum particles are naturally confined to stay into $M$. 
This differs from what happens, for example, in the case of the Euclidean Laplacian on $\R^2\setminus\mathcal Z$. Indeed, this operator is not essentially self-adjoint and its different self-adjoint extensions correspond to different dynamics, e.g.\ to complete reflection or transmission  of quantum particles at $\mathcal Z$, to be chosen depending on the physics of the problem. Similar considerations hold for heat diffusion or wave equations.

The Grushin structure belongs to a class of singular Riemannian structures, called almost-Riemannian structure (ARS), introduced in \cite{ABS-Gauss-Bonnet}. The study of essential self-ad\-joint\-ness of the Laplace-Beltrami operator for ARS has been initiated in \cite{BL-LaplaceBeltrami, anticonic}, for surfaces, and in \cite{quantum-confinement}, for general dimension. In the latter, as a particular instance of a more general criterion, it has been proved that the metric boundary of a non-complete Riemannian manifold can develop a repulsive effect, quantified in terms of an intrinsic invariant called \emph{effective potential}, whose strength can entail the essential self-adjointness of the Laplace-Beltrami operator \cite[Thm.\ 1]{quantum-confinement}.


In this paper we extend the results of \cite{quantum-confinement} to a class of natural (H\"ormander-type) hypoelliptic operators, the \emph{sub-Laplacians}, arising in sub-Riemannian geometry as a generalization of the Riemannian Laplace-Beltrami operator to this setting. 

Roughly speaking, a sub-Riemannian structure on a smooth manifold $N$ is defined by a (possibly rank-varying) smooth distribution $\distr\subset TN$ endowed with a scalar product $g:\distr\times\distr\to \R$. (For a precise definition, see Section~\ref{sec:preliminaries}.)
Since the distribution $\distr$ is assumed to satisfy the \emph{Lie bracket generating condition}, any two points in $N$ can be joined by curves a.e.\ tangent to $\distr$, of which the scalar product allows to measure the length.
As in the Riemannian case, by minimizing the length of such curves one can define a distance $d$ on $N$.
Given a measure $\omega$ on $N$, which is smooth outside of some closed set $\mathcal Z\subset N$, the associated sub-Laplacian is the H\"ormander-type operator on $L^2(N,\omega)$ defined by
\begin{equation}
  \Delta_{\omega} = \dive_{\omega} \circ \nabla, \qquad \dom(\Delta_{\omega}) = C^\infty_c(N\setminus \mathcal Z),
\end{equation}
where the divergence is computed with respect to $\omega$, and $\nabla$ is the sub-Riemannian gradient.

It is well known that if $\mathcal Z = \varnothing$ and the sub-Riemannian structure is complete then $\Delta_\omega$ is essentially self-adjoint on $L^2(N,\omega)$ \cite{Strichartz}. Here, we focus on the case of singular measures $\omega$, that is $\mathcal{Z} \neq \varnothing$. In this setting, our main result is the following criterion for essential self-adjointness of sub-Laplacians, that generalizes \cite[Thm.\ 1]{quantum-confinement}. 

\begin{theorem}\label{t:a}
Let $N$ be a complete sub-Riemannian manifold endowed with a measure $\omega$. Assume $\omega$ to be smooth on 
$N\setminus \mathcal{Z}$, where the singular set $\mathcal Z$ is a smooth, embedded, compact hypersurface with no characteristic points.
Assume also that, for some $\varepsilon>0$, there exists a constant $\kappa\geq 0$ such that, letting $\delta = d(\mathcal{Z}, \cdot\,)$, we have
\begin{equation}\label{eq:cond-intro}
\Veff = \left(\frac{\Delta_\omega\delta}{2}\right)^2 + \left(\frac{\Delta_\omega\delta}{2}\right)^\prime \geq \frac{3}{4\delta^2} -\frac{\kappa}{\delta}, \qquad \text{for } 0<\delta \leq \varepsilon,
\end{equation}
where the prime denotes the derivative in the direction of $\nabla\delta$. 
Then, $\Delta_\omega$ with domain $C^\infty_c(M)$ is essentially self-adjoint in $L^2(M)$, where $M = N \setminus \mathcal{Z}$, or any of its connected components.

Moreover, if $M$ is relatively compact, the unique self-adjoint extension of $\Delta_\omega$ has compact resolvent. Therefore, its spectrum is discrete and consists of eigenvalues with finite multiplicity.
\end{theorem}
\begin{rmk}
The compactness of $\mathcal{Z}$ in Theorem~\ref{t:a} is not necessary, and it can be replaced by the weaker assumption that the (normal) injectivity radius from $\mathcal{Z}$ is strictly positive.
\end{rmk}

We stress that, although the blueprint for the proof of Theorem~\ref{t:a} follows the idea of \cite{quantum-confinement}, it is not a straightforward adaptation. Indeed, the new aspects of the proof are the exploitation of subellipticity to obtain regularity properties of weak solutions (Lemma \ref{l:shubin}), the sub-Riemannian version of the Rellich-Kondrachov theorem (Lemma \ref{l:RK}), and a sub-Riemannian tubular neighborhood theorem for smooth hypersurfaces with no characteristic points (Proposition~\ref{prop:smoothness}). { We believe that these results are interesting on their own. In particular, up to our knowledge, Proposition~\ref{prop:smoothness} is the first tubular neighborhood result holding for general sub-Riemannian manifolds. See, e.g., \cite{R17, AFM17}, where results of this type are proved for (possibly higher codimensional) submanifolds in some Carnot groups.}

A particularly interesting case, is the one where the measure $\omega$ is chosen to be the \emph{Popp's measure} $\mathcal P$. This is a measure canonically associated with the sub-Riemannian structure, which is smooth where the structure is equiregular \cite{montgomerybook,BarilariRizzi13}. In this case, the singular region $\mathcal{Z}$ coincides with the singular region of the sub-Riemannian structure, i.e., the complement of the equiregular region. We refer to the sub-Laplacian $\Delta_{\mathcal P}$ associated with $\mathcal P$ as the intrinsic (or Popp) sub-Laplacian. 

Consider for example the Martinet structure on $N = \R^3$, whose distribution and metric are defined by the orthonormal vector fields
\begin{equation}
X_1 = \partial_y + {x^2}\partial_z, \qquad X_2 = \partial_x.
\end{equation}
The distribution $\distr = \spn\{X_1,X_2\}$ is then equiregular everywhere except on the hypersurface $\mathcal{Z} = \{ x =0 \}$, where Popp's measure is singular. Indeed,
\begin{equation}
\mathcal{P} = \frac{1}{2\sqrt{2}|x|} dx \wedge dy \wedge dz.
\end{equation}
In this case, $\Delta_{\mathcal{P}}$ with domain $C^\infty_c(N\setminus\mathcal Z)$ is essentially self-adjoint. This fact has been proved in \cite[Thm.\ 3]{BL-LaplaceBeltrami} for a compactified version of the Martinet structure on $\R\times \mathbb{S}^1 \times \mathbb{S}^1$, using a Fourier decomposition w.r.t.\ the compact singular region $\mathcal Z \simeq \mathbb S^1\times \mathbb S^1$. 

This result has driven the authors to conjecture that the loss of equiregularity acts as a general barrier for quantum diffusion, i.e., more precisely, that the intrinsic sub-Laplacian, when restricted to the equiregular region of a sub-Riemannian manifold, is essentially self-adjoint. As a consequence of Theorem~\ref{t:a}, we prove this conjecture under mild regularity assumptions on the sub-Riemannian structure (Popp-regularity, see Section~\ref{sec:examples}).

\begin{theorem}\label{t:b}
Let $N$ be a complete and Popp-regular sub-Riemannian manifold, with compact singular set $\mathcal Z$. Then, the sub-Laplacian $\Delta_\mathcal P$ with domain $C_c^\infty(M)$ is essentially self-adjoint in $L^2(M)$, where $M=N\setminus \mathcal Z$ or one of its connected components. Moreover, if $M$ is relatively compact, the unique self-adjoint extension of $\Delta_\mathcal P$ has compact resolvent.
\end{theorem}

\begin{rmk}
 The Popp sub-Laplacian is not the only intrinsic second order diffusion operator associated with a sub-Riemannian structure. See for example \cite{Random1,Random2,Random3} for intrinsic operators associated with sub-Riemannian random walks. Other possible sub-Laplacians are related with different choices of intrinsic measures. For example, one might consider the Hausdorff measure $\mathcal H$ associated with the metric structure. However, with the exception of some low dimensional cases, it is unknown whether $\mathcal H$ is even $C^1$, see \cite{ABB12}. Since Theorem~\ref{t:a} requires smooth measures (actually, $C^2$ is sufficient) we have restricted our attention to the Popp measure.
\end{rmk}

\subsection{Structure of the paper}
The necessary preliminaries of  sub-Riemannian geometry are discussed in Section~\ref{sec:preliminaries}. Section~\ref{sec:distance-singular-region} is devoted to the proof of regularity properties of the distance function from the singular region. In Section~\ref{sec:main-criterion} we prove Theorems~\ref{t:a} and~\ref{t:b}. In Section~\ref{sec:examples} we discuss the case of the intrinsic sub-Laplacian. We close the paper with examples of non-Popp-regular structures where Theorem~\ref{t:b} does not apply, but Theorem~\ref{t:a} does, and hence the intrinsic sub-Laplacian is essentially self-adjoint. We also provide examples where both results do not apply, and we are not able to determine whether the sub-Laplacian is essentially self-adjoint.

\section{Preliminaries on sub-Riemannian geometry}
\label{sec:preliminaries}
\begin{definition}
Let $N$ be a connected smooth manifold. A \emph{sub-Riemannian structure} on $N$ is a triple $\big(U,\xi,(\cdot|\cdot)_q\big)$, where
\begin{itemize}
\item $\pi_U:U\to N$ is an Euclidean bundle  with base $N$ and Euclidean fiber $U_q=\pi^{-1}(q)$, in particular for every $q \in N$, $U_q$ is a vector space equipped with a scalar product $(\cdot | \cdot)_q$, smooth with respect to $q$. 
\item $\xi : U \to TN$ is a vector bundle morphism, i.e., $\xi$ is a fiber-wise linear map such that, letting $\pi : T N \to N$ be the canonical projection, the following diagram commutes:
\begin{equation*}
\xymatrix{
U \ar[dr]_{\pi_{U}} \ar[r]^{\xi}
& TN \ar[d]^{\pi} \\
 & N }
\end{equation*}
\item The \emph{Lie bracket generating condition} holds true, i.e.,
\begin{equation}
\label{eq:bracketc}
\left.\mathrm{Lie}(\xi(\Gamma(U)))\right|_q=T_qN, \qquad \forall q\in N,
\end{equation}
where $\Gamma(U)$ denotes the $C^{\infty}(N)$-module of smooth sections of $U$ and $\left.\mathrm{Lie}(\xi(\Gamma(U)))\right|_q$ denotes the smallest Lie algebra containing $\xi(\Gamma(U))\subseteq\Gamma(TN)$, evaluated at $q$.
\end{itemize}
The subspace of \emph{horizontal directions} at $q\in N$ is $\distr_q=\xi(U_q)\subseteq T_qN$ and the set of \emph{horizontal vector fields} is $\Gamma(\mathcal D)=\xi(\Gamma(U))$.
\end{definition}

Consider a local frame for $U$, i.e., a set $\{\sigma_1,\ldots,\sigma_m\}$, with $m=\rank(U)$, of smooth local sections of $U$, defined on some neighborhood $\mathcal{O} \subseteq N$, and  which are orthonormal with respect to the scalar product on $U$. The vector fields $X_i:= \xi \circ \sigma_i$ constitute a \emph{local generating family}. On $\mathcal O$, condition \eqref{eq:bracketc} reads
\begin{equation}
\left.\mathrm{Lie}(X_1,\dots, X_m)\right|_{q}=T_qN,\qquad \forall q\in\mathcal O.
\end{equation}

Let $r(q)=\dim(\mathcal D_q)$ be the \emph{rank} of the distribution at $q \in N$. Moreover, for $k\in\N$, let
\begin{equation}
\mathcal D^k_q:=\spn\{[X_1,\dots,[X_{j-1},X_j]]_q : X_i\in\Gamma(\mathcal D),\ j\leq k\}.
\end{equation}
By \eqref{eq:bracketc}, we call the \emph{step} of the sub-Riemannian structure at $q$ the minimal integer $s=s(q)\in\N$ such that
$\mathcal D^s_q=T_qN$.
\begin{definition} 
Let $A\subseteq N$. We say that a sub-Riemannian structure on $N$ is \emph{equiregular} on $A$ if $\dim(\mathcal D^k_q)$ is constant for $q\in A$ and for any $k\in\N$.
\end{definition}
Notice that even $r(q)=\dim(\mathcal D^1_q)$ can be non-constant. 
For instance, this is the case of almost-Riemannian manifolds, where  there exists a closed set $\mathcal Z\subset N$ such that $\dim(\mathcal D^1_q)=\dim N$ for every $q\in N\setminus \mathcal Z$.

\medskip

In this paper, $N$ is a smooth manifold without boundary, endowed with a sub-Rie\-mann\-ian structure. Moreover, we let $\mathcal Z\subset N$ be a set satisfying
\begin{equation}
\label{eq:singular_set}
\tag{H0}
\mathcal Z\subseteq N\text{ is a smooth, embedded hypersurface.}
\end{equation}
The set $\mathcal Z$ will be called the \emph{singular region} when defined in association with a measure $\omega$ on $N$, smooth on $N\setminus \mathcal Z$.

\begin{definition}
Let $\mathcal Z\subseteq N$ be  a smooth embedded hypersurface. We say that $q\in\mathcal Z$ is a  \emph{characteristic (or tangency) point} if $\mathcal D_q\subseteq T_q\mathcal Z$.
\end{definition}
\noindent We will also assume that:
\begin{equation}
\label{eq:no_characteristic}
\tag{H1} 
\text{The singular region $\mathcal Z$ does not contain characteristic points.}
\end{equation}

Assumption \eqref{eq:no_characteristic} implies that there are no abnormal minimizers between $p\in N\setminus\mathcal Z$ and $\mathcal Z$ (see Proposition \ref{prop:no_abn}). However, we do not exclude the presence of other abnormal minimizers. (See \cite{nostrolibro} for a definition of abnormal minimizers.)

\subsection{Metric structure} 
Let $q\in N$ and $v\in \mathcal D_q$. We define the \emph{sub-Riemannian norm} as 
\begin{equation}
|v|^2=\inf\{(u|u)_q : u\in U_q,\ \xi(u)=v\}.
\end{equation}
One can check that the above norm satisfies the parallelogram law, and hence it is defined by a scalar product on $\mathcal D_q$, denoted with the symbol $g_q$.
 
An {\em horizontal curve} is an absolutely continuous curve $\gamma:[0,1]\to N$ such that there exists an $L^1$ curve $\eta : [0,1] \to U$ satisfying $\pi_U(\eta) = \gamma$ and
\begin{equation}
\dot\gamma(t) = \xi(\eta(t)), \qquad \text{for a.e. } t \in [0,1].
\end{equation}
In particular, $\dot\gamma(t) \in \distr_{\gamma(t)}$ for a.e.\ $t \in [0,1]$. In this case, we define the \emph{length} of $\gamma$ as
\begin{equation}
\ell(\gamma)=\int_0^1 |\dot\gamma(t)|\;dt.
\end{equation}
Since $\ell$ is invariant by reparametrization of $\gamma$, when dealing with minimization of length we consider only intervals of the form $[0,1]$.
We define the \emph{sub-Riemannian distance} as
\begin{equation}
d(p,q):=\inf\{\ell(\gamma) : \gamma\text{ is horizontal, } \gamma(0)=p,\ \gamma(1)=q\}.
\end{equation}
Under the bracket-generating condition \eqref{eq:bracketc}, the Chow-Rashevskii Theorem implies that any couple of points $p,q\in N$ can be connected by means of horizontal curves. That is, $d:N\times N\to\R$ is finite. Moreover, $d$ is a continuous map and the metric space $(N,d)$ has the same topology as $N$.

\begin{definition}
The sub-Riemannian (or horizontal) gradient of a smooth function $f$ is the smooth vector field $\nabla f \in \Gamma(\distr)$ such that
\begin{equation}
\label{eq:gradient}
g(\nabla f,W)= d f (W), \qquad \forall W \in \Gamma(\distr).
\end{equation}
\end{definition}

\begin{rmk}
In terms of a local generating family $X_1,\ldots,X_r$ for the sub-Riemannian structure, we have
\begin{equation}
\label{eq:grad}
\nabla f = \sum_{i=1}^r X_i(f) X_i, \qquad |\nabla f |_{}^2 = \sum_{i=1}^r X_i(f)^2.
\end{equation}
Formula \eqref{eq:grad} holds also if $X_1,\ldots,X_r$ are not independent, in particular it holds on $\mathcal Z$.
\end{rmk}

\subsubsection{Sub-Laplacians}
Let $\omega$ be a measure on $N$, smooth and positive on $N \setminus \mathcal{Z}$. The sub-Laplacian $\Delta_\omega$ is the operator
\begin{equation}
\Delta_\omega u := \dive_\omega (\nabla u), \qquad \forall u \in C^\infty_c(N \setminus \mathcal{Z}),
\end{equation}
where the divergence $\dive_\omega$ is computed with respect to the measure $\omega$, and $\nabla$ is the sub-Riemannian gradient. Equivalently, $\Delta_\omega$ can be defined as the operator associated with the quadratic form \begin{equation}
\mathcal{E}(u,v) := \int_M g(\nabla u, \nabla \bar{v})\, d\omega, \qquad \forall u,v \in C^\infty_c(N\setminus \mathcal{Z}).
\end{equation}
In terms of a local generating family of vector fields $X_1,\dots,X_r \subset \Gamma(\distr)$, we have
\begin{equation}\label{eq:sublaplc}
\Delta_{\omega} = \sum_{i=1}^k X_i^2 + \dive_{\omega} (X_i)X_i.
\end{equation}
As a consequence of the Lie bracket generating assumption, $\Delta_\omega$ is hypoelliptic \cite{Hormander}. Finally, it is well-known that if $\mathcal Z = \varnothing$ and the sub-Riemannian structure is complete then $\Delta_\omega$ is essentially self-adjoint on $L^2(N)$ \cite{Strichartz}.

\subsubsection{Geodesics and Hamiltonian flow}

We recall basic notions on minimizing curves in sub-Riemannian geometry. A \emph{geodesic} is a horizontal curve $\gamma :[0,1] \to N$ that locally minimizes the length between its endpoints, and is parametrized by constant speed.
\begin{definition}
The  \emph{sub-Riemannian Hamiltonian} is the smooth function $H : T^*N \to \R$,
\begin{equation}
\label{eq:Hamiltonian}
H(\lambda) := \frac{1}{2}\sum_{i=1}^r \langle \lambda, X_i \rangle^2, \qquad \lambda \in T^*N,
\end{equation}
where $X_1,\ldots,X_r$ is a local generating family for the sub-Riemannian structure, and $\langle \lambda, \cdot \rangle $ denotes the action of covectors on vectors. 
Associated with $H$ we define the \emph{Hamiltonian vector field} $\vec H$ on $T^*N$ as  $\vec H: C^\infty(T^*N)\to C^\infty(T^*N)$ such that $\sigma(\cdot,\vec H)=dH$. Here, $\sigma\in\Lambda^2(T^*N)$ is the canonical symplectic form on $T^*N$.
\end{definition}
Solutions $\lambda : [0,1] \to T^*N$ of \emph{Hamilton equations}
\begin{equation}\label{eq:Hamiltoneqs}
\dot{\lambda}(t) = \vec{H}(\lambda(t))
\end{equation}
are called \emph{normal extremals}. Their projections $\gamma(t) := \pi(\lambda(t))$ on $N$, where $\pi:T^*N\to N$ is the canonical projection, are locally minimizing curves parametrized by constant speed, and are called \emph{normal geodesics}.
It is easy to show that if $\lambda(t)$ is a normal extremal, and $\gamma(t) = \pi(\lambda(t))$ is the corresponding normal geodesic, then
\begin{equation}
\dot{\gamma}(t) = \sum_{i=1}^r \langle\lambda(t),X_i(\gamma(t))\rangle X_i(\gamma(t)),
\end{equation}
and its speed is given by $| \dot\gamma | = \sqrt{2H(\lambda)}$. In particular 
\begin{equation}
\label{eq:speed}
\ell(\gamma|_{[0,t]}) = t \sqrt{2H(\lambda(0))}\quad \forall t\in[0,1].
\end{equation} 
\begin{definition}
The \emph{exponential map} $\exp_{q} : \distr_q \to N$, with base $q \in N$ is
\begin{equation}
\exp_q (\lambda) := \pi \circ e^{\vec{H}}(\lambda), \qquad \lambda \in \distr_q,
\end{equation}
where $\distr_q \subseteq T_q^* N$ is the set of covectors such that the solution $t \mapsto e^{t \vec{H}}(\lambda)$ of~\eqref{eq:Hamiltoneqs} with initial datum $\lambda$ is well defined up to time $1$.
\end{definition}
We say that a sub-Riemannian structure on $N$ is complete if $(N,d)$ is a complete metric space. In a complete sub-Riemannian structure, the sub-Riemannian version of Hopf-Rinow theorem implies that $\distr_q=T_q^*N$ for every $q\in N$.

There is another class of minimizing curves in sub-Riemannian geometry, called \emph{abnormal minimizers}. These curves can still be lifted to extremal curves $\lambda(t)$ on $T^*N$, but which may not follow the Hamiltonian dynamic of \eqref{eq:Hamiltoneqs}. 
Here we only observe that an extremal $\lambda(t)\in T^*N$ is abnormal if and only if it satisfies:
\begin{equation}
\label{eq:abn}
\langle \lambda(t),\distr_{\pi(\lambda(t))}\rangle=0\qquad \forall t\in[0,1] 
\end{equation}
with $\lambda(t)\neq 0$ for any $t\in[0,1]$ (see \cite[Thm 3.53]{nostrolibro}), that is $H(\lambda(t))\equiv 0$.
Notice also that a curve may be abnormal and normal at the same time.

\begin{prop}\label{prop:no_abn}
Consider a sub-Riemannian structure on a smooth manifold $N$. Let $\mathcal{Z} \subset N$ be a closed embedded hypersurface. Let $\gamma:[0,1]\to N$ be a minimizer such that $\gamma(0)\in \mathcal Z$, $\gamma(1)=p\in N\setminus \mathcal{Z}$ and
\begin{equation}
\ell(\gamma)=\inf\{d(q,p),\ q\in \mathcal Z\}.
\end{equation}
Then $\gamma(0)\in \mathcal Z$ is a characteristic point if and only if $\gamma$ is abnormal.
\end{prop}
\begin{proof}
By assumption, there exists an extremal $t\in[0,1]\mapsto\lambda(t)$ such that $\gamma(t)=\pi(\lambda(t))$. On the other hand, $\gamma$ minimizes also the distance from $\mathcal Z$, hence, by \cite[Thm 12.4]{AS-GeometricControl} the following  transversality condition holds true:
\begin{equation}\label{eq:trans}
\langle \lambda(0), v\rangle=0,\qquad \forall v\in T_{\gamma(0)}\mathcal Z.
\end{equation}
If $\gamma(0)$ is a characteristic point, the above implies $H(\lambda(0))=0$. Hence, $\lambda(t)$ cannot be normal, and is thus abnormal. On the other hand, if $\lambda(t)$ is abnormal, it satisfies $H(\lambda(0))=0$, that is \eqref{eq:abn}. We deduce that $\gamma(0)$ is a characteristic point. In fact, if 
$\distr_{\gamma(0)}$ were transversal to $T_{\gamma(0)}\mathcal Z$,
\eqref{eq:trans} and \eqref{eq:abn} 
would imply $\lambda(0)= 0$ yielding a contradiction.
\end{proof}

\subsection{Popp's measure}\label{ss:popp}

 On equiregular neighborhoods of a sub-Riemannian manifold, it is possible to define an intrinsic smooth measure $\mathcal P$, called Popp's measure.
 This measure was introduced first in \cite{montgomerybook} and then used in \cite{ABGR2009} to define an intrinsic sub-Laplacian in the sub-Riemannian setting. In the following, we recall the explicit formula for Popp's measure given in \cite{BarilariRizzi13} in terms of adapted frames, which will be used in Section~\ref{sec:examples}. 

Let $\mathcal O\subseteq N$ be an equiregular neighborhood of an $n$-dimensional sub-Riemannian manifold $N$. A local frame $X_1,\dots,X_n$ on $\mathcal{O}$ is said to be \emph{adapted} to the sub-Riemannian structure if $X_1,\dots,X_{k_i}$ is a local frame for $\mathcal D^i$, where $k_i=\dim(\mathcal D^i)$ is constant on $\mathcal O$. In particular $r(q)\equiv r$ is constant on $\mathcal O$. Notice that, the equiregularity assumption means that, on $\mathcal O$, $\distr^i$ are ``true'' distributions, and hence that there always exists a local adapted frame.
Define the smooth functions $b_{i_1\dots i_j}^\ell\in C^\infty(N)$ as
\begin{equation}
[X_{i_1},[X_{i_2},\dots,[X_{i_{j-1}},X_{i_j}]]]=\sum_{\ell=k_{j-1}+1}^{k_j} b_{i_1i_2\dots i_j}^\ell X_\ell \quad \mod \mathcal\ \mathcal D^{j-1},
\end{equation}
where $1\leq i_1,\dots,i_{j}\leq m=\dim(\mathcal D^1)$. Consider the $k_j-k_{j-1}$ dimensional square matrices
\begin{equation}
(B_j)^{h\ell}=\sum_{i_1,\dots,i_j=1}^r b_{i_1,\dots,i_j}^hb_{i_1,\dots,i_j}^\ell,\qquad \forall j=1,\dots,s,
\end{equation}
where $s$ is the step of the structure.
Then, denoting by $\nu^1,\dots,\nu^n$ the dual frame to $X_1,\dots,X_n$, the Popp's measure reads
\begin{equation}
\label{eq:popp}
\mathcal P=\frac{1}{\sqrt{\prod_{j=1}^s \det B_j}}|\nu^1\wedge\dots\wedge\nu^n|.
\end{equation}

One can check that the measure defined by \eqref{eq:popp} does not depend on the choice of the local adapted frame, and can be taken as the definition of Popp's measure. It is not hard to see, using the very definition, that if $q\in \bar{\mathcal{O}}$ is a non equiregular point, then $\lim\sqrt{\prod \det B_j}=0$, and hence 
the Radon-Nikodym derivative of Popp's measure computed with respect to any globally smooth measure on $N$ diverges to $+\infty$ on the singular region $\mathcal Z$.  Uniform estimates of this divergence can be found in \cite{GhezziJean17}.

\section{Sub-Riemannian distance from an hypersurface}
\label{sec:distance-singular-region}
We recall that $N$ is a smooth  (connected) manifold endowed with a sub-Riemannian structure, and that $\mathcal Z\subset N$ is a closed, embedded hypersurface with no characteristic points. We stress that $\mathcal{Z}$ is not necessarily the complement of the equiregular region of the sub-Rie\-mann\-ian structure. The distance from the singular region $\delta:N\to[0,\infty)$ is
\begin{equation}
\delta(p)=\inf\{d(q,p) \mid q\in \mathcal Z\},\qquad \forall p\in N.
\end{equation}
In the following we resume some fundamental facts about $\delta$. (See Figure~\ref{fig:1}.) 
\begin{figure}
\includegraphics[scale=.9]{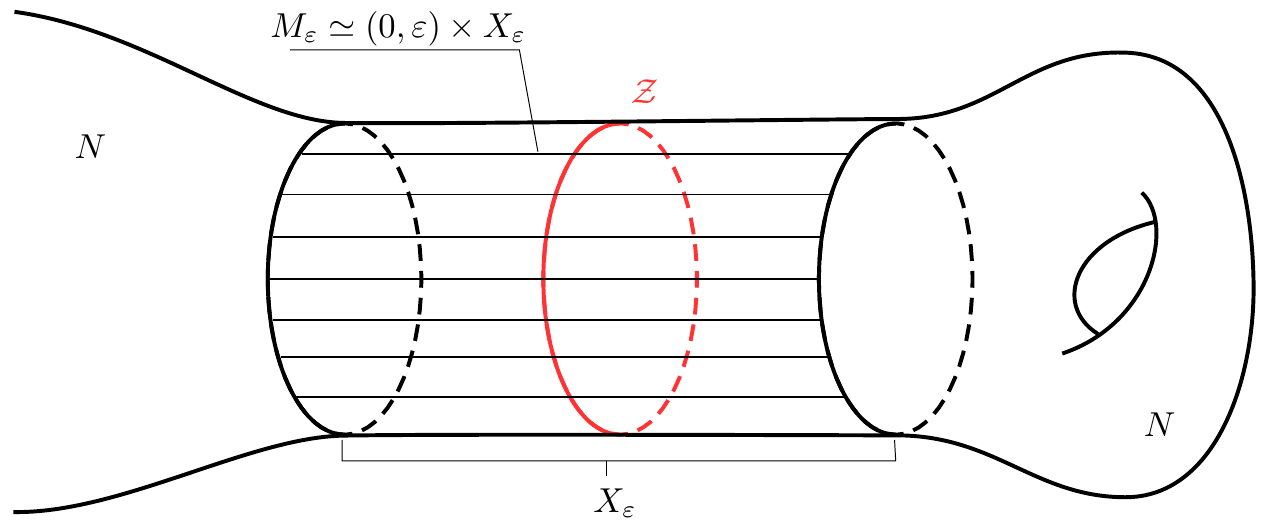}
\caption{Tubular neighborhood of the singular region.}\label{fig:1}
\end{figure}

\begin{prop}
\label{prop:smoothness}
Let $N$ be a smooth sub-Riemannian manifold and 
$\mathcal Z\subset N$ be a smooth, embedded, compact hypersurface with no characteristic points.
Then:
\begin{itemize}
\item[i)] $\delta:N\to[0,\infty)$ is Lipschitz w.r.t.\ the sub-Riemannian distance and $|\nabla\delta|\leq 1$ a.e.;
\item[ii)] there exists $\varepsilon>0$ such that $\delta:M_\varepsilon\to[0,\infty)$ is smooth, where $M_\varepsilon=\{ 0<\delta(p)<\varepsilon\}$;
\item[iii)] letting $X_\varepsilon=\{\delta(p)=\varepsilon\}$,
there exists a smooth diffeomorphism 
$F: (0,\varepsilon)\times X_\varepsilon\to M_{\varepsilon}$,
such that 
\begin{equation}
\delta(F(t,q))=t\quad\text{and}\quad F_*\partial_t=\nabla\delta,\quad \text{for }(t,q)\in (0,\varepsilon)\times X_\varepsilon.
\end{equation}
Moreover, $|\nabla \delta|\equiv 1$ on $M_\varepsilon$.

\end{itemize}
\end{prop}
\begin{rmk}
The statement and the proof can be simplified if $\mathcal Z$ is two-sided (e.g. when $N$ and $\mathcal Z$ are orientable). In this case, $M_\varepsilon=(-\varepsilon,0)\times \mathcal Z \sqcup (0,\varepsilon)\times \mathcal Z$ and there is no need to introduce $X_\varepsilon$. However, this is not true if $\mathcal Z$ is one-sided. (For example, think at a Grushin-like structure on the M\"obius strip, where $\mathcal Z$ is the central line.)
\end{rmk}

\begin{proof} 
We prove i). Let $p,q\in N$. By the triangle inequality we have $\delta(p)\leq d(p,q)+\delta(q)$, thus proving that $\delta$ is $1$-Lipschitz with respect to the sub-Riemannian distance. By \cite[Thm.\ 8]{FHK99} (see also \cite[Prop.\ 2.9]{FSSC97}, \cite[Thm.\ 1.3]{GN98}) this implies that the sub-Riemannian gradient satisfies $|\nabla\delta|\leq 1$ almost everywhere.

To prove ii), we follow the same strategy presented in \cite[Lemma 7.7
]{quantum-confinement}.
We first define the {\em annihilator bundle of the singular set} 
\begin{equation}
\label{eq:annihil}
A\mathcal{Z} := \{(q,\lambda) \in T^*N \mid \lambda(T_q \mathcal{Z}) = 0 \},
\end{equation}
which is a rank 1 vector bundle with base $\mathcal{Z}$.  The map $i_0  : \mathcal{Z} \to A\mathcal{Z}$, $i_0(q) = (q,0)$ is an embedding of $\mathcal{Z}$ onto the zero section of $A\mathcal{Z}$. 
The bundle $A\mathcal Z$ plays the role of the Riemannian normal bundle usually employed in the construction of  a tubular neighborhood.  
Let $0 \neq \lambda \in A_q\mathcal{Z}$. Since $q$ is not a characteristic point, we have $\lambda(\distr_q) \neq 0$. Hence $H(\lambda) >0$,
and the vector 
\begin{equation}
v_\lambda=\pi_*\vec{H}(\lambda)
=\sum_{i=1}^r\langle \lambda, X_i\rangle X_i(q),
\end{equation}
where $X_1,\dots, X_r$ is a local generating frame of $\mathcal D$, is a non-zero horizontal vector transversal to $T_q\mathcal Z$. Observe that $|v_\lambda|^2=\langle\lambda,v_\lambda\rangle=2H(\lambda)>0$, even if $X_1,\dots,X_r$ are not independent at $q$.

Let $D \subseteq T^*N$ be the set of $(q,\lambda)$ such that $\exp_q(\lambda)$ is well defined. Indeed, $D$ is open and so is $D \cap A\mathcal{Z}$ as a subset of $A\mathcal Z$. Consider the map $E : A \mathcal{Z} \cap D \to N$, given by
\begin{equation}
\label{eq:E}
E(q,\lambda) := \exp_q(\lambda) = \pi\circ e^{\vec{H}}(\lambda).
\end{equation}

{\em Claim 1.}
Given $q\in\mathcal Z$, $E$ is a diffeomorphism on a neighborhood  $U(q)\subseteq D\cap A\mathcal Z$ of $i_0(q)=(q,0) \in A \mathcal{Z}$. \\
To prove Claim 1, we first notice that $i_0(\mathcal{Z}) \subseteq D$, and $E\circ i_0 = \mathrm{id}_{\mathcal{Z}}$.
Moreover, $dE$ has full rank on $i_0(\mathcal{Z})$. In fact, identifying $T_{(q,0)}A\mathcal Z\simeq T_q\mathcal Z\oplus A_q\mathcal Z$, we have $d_{(q,0)}E|_{T_q\mathcal Z}=\mathrm{id}_{T_q\mathcal Z}$ and
for $\delta\lambda\in A_q\mathcal Z$
\begin{equation}
\begin{split}
d_{(q,0)}E(\delta\lambda)&
=\sum_{i=1}^r\langle\delta \lambda,X_i\rangle X_i= v_{\delta\lambda}\neq 0.
\end{split}
\end{equation}
Claim $1$ now follows from the inverse function theorem and from the fact that $\dim(A\mathcal{Z}) = \dim(N)$. Moreover, since $\mathcal{Z}$ is embedded, and $2H$, restricted to the fibers of $A\mathcal{Z}$, is a well defined norm, the neighborhood $U(q)$ can be taken of the form
\begin{equation}
U(q)=U_{\varrho}(q)= \{(q',\lambda') \mid d  (q,q') < \varrho,\; \sqrt{2H(\lambda')} < \varrho\}, \qquad \varrho>0.
\end{equation}
For any $q \in \mathcal{Z}$, let 
\begin{equation}
\varepsilon(q) := \sup \{ \varrho>0 \mid E : U_{\varrho}(q) \to E(U_{\varrho}(q)) \text{ is a diffeomorphism} \} > 0.
\end{equation}

{\em Claim 2.} The function $\varepsilon : \mathcal{Z} \to \R_+$ is continuous, since
\begin{equation}\label{eq:ineqs}
|\varepsilon(q)-\varepsilon(q')| \leq d (q,q'), \qquad \forall q,q' \in \mathcal{Z}.
\end{equation}
To prove it, assume without loss of generality that $\varepsilon(q) \geq \varepsilon(q')$. If $d (q,q') \geq \varepsilon(q)$, then \eqref{eq:ineqs} holds. On the other hand, if $d (q,q') < \varepsilon(q)$, the triangle inequality for $d $ implies that that $U_{\varrho}(q') \subseteq U_{\varepsilon(q)}(q)$ for $\varrho =\varepsilon(q) - d (q,q')$, implying Claim 2.

Thanks to the compactness\footnote{In view of Remark~\ref{rmk:cpt}, we notice that the function $\varepsilon(q)$ is the sub-Riemannian version of the normal injectivity radius from $\mathcal{Z}$ at $q$, and thus $\inf_q\varepsilon(q)$ is the normal injectivity radius from $\mathcal{Z}$. Hence, if $\mathcal{Z}$ is not compact, we can still proceed by assuming that the normal injectivity radius from $\mathcal{Z}$ is strictly positive.} of $\mathcal{Z}$, we define the open neighborhood of $i_0(\mathcal{Z})$:
\begin{equation}
U:=\{(q,\lambda) \in A\mathcal{Z} \mid \sqrt{2H(\lambda)} < \varepsilon_0 \}, \qquad \varepsilon_0 := \min \{\varepsilon(q)/2 \mid q \in \mathcal{Z}\}>0.
\end{equation}

{\em Claim 3}. The restriction of $E$ to $U$ is injective. \\
This follows from the fact that for $(q_1,\lambda_1), (q_2,\lambda_2)\in U$, if $\varepsilon(q_1) \leq \varepsilon(q_2)$, then $(q_1,\lambda_1)\in U_{\varepsilon (q_2)}(q_2)$ (on which $E$ is a diffeomorphism by Claim 1).

By Claim 3, $E: U \to E(U)$ is a smooth diffeomorphism and  $E(U) \subseteq \{ \delta  < \varepsilon_0\}$. Up to taking a smaller $\varepsilon_0$, we can assume that $E(U) \subseteq \{ \delta  < \varepsilon_0\} \subseteq K$, where $K$ is compact.
\vspace{0.1cm}

{\em Claim 4.}
 $E(U) = \{ \delta  < \varepsilon_0\}$ and, on $E(U)$, the sub-Riemannian distance from $\mathcal{Z}$ satisfies
\begin{equation}\label{eq:formuladistance}
\delta (E(q,\lambda)) = \sqrt{2H(\lambda)}.
\end{equation}

To prove Claim 4, let $p \in \{ \delta  < \varepsilon_0\} \subseteq K$. Since $K$ is compact, there exists at least one horizontal curve $\gamma:[0,1] \to N$ minimizing the sub-Riemannian distance between $\mathcal{Z}$ and $p$. By Proposition \ref{prop:no_abn}, this must be a normal geodesic, that is $p=E(q,\lambda)$, with $q \in \mathcal{Z}$ and $\lambda \in T_{q}^*N$. 
Since $\gamma$ is minimizing, transversality conditions \eqref{eq:trans} imply that $\lambda(T_{q} \mathcal{Z}) = 0$, that is $(q,\lambda) \in A\mathcal{Z}$. Moreover, $\sqrt{2H(\lambda)} = \ell(\gamma) =  \delta (p) < \varepsilon_0$. This implies that $(q,\lambda) \in U$, that is $p = E(q,\lambda) \in E(U)$, and $\delta (E(q,\lambda)) = \sqrt{2H(\lambda)}$, as claimed.
Since $\sqrt{2H(\lambda)}$ is a smooth function for $H(\lambda)\neq 0$, $\delta$ is smooth on $\{0< \delta  < \varepsilon\}$, for all $\varepsilon \leq \varepsilon_0$.

We prove statement iii). 
Let $0<\varepsilon<\varepsilon_0$
and let $F:(0,\varepsilon)\times X_\varepsilon\to M_\varepsilon$
be defined by
\begin{equation}
F(t,q)=E\left(q_0,\frac{t}{\sqrt{2H(\lambda)}}\lambda\right)
\end{equation}
where, for $q\in X_\varepsilon$, we are using Claim 4 to write $q=E(q_0,\lambda)$ for a unique $(q_0,\lambda)\in U$ such that $\sqrt{2H(\lambda)}=\varepsilon$.
The function $F$ is a smooth diffeomorphism, with inverse
\begin{equation}
F^{-1}(p)=\left(\sqrt{2H(\nu)},E\Big(p_0,\frac{\varepsilon}{\sqrt{2H(\nu)}}\nu\Big)\right),\text{ for } p=E(p_0,\nu)\in M_\varepsilon,\ (p_0,\nu)\in U.
\end{equation}
Moreover, by \eqref{eq:formuladistance} and the definition of $F$
\begin{equation}
\delta(F(t,q))=\sqrt{2H\left(\frac{t}{\sqrt{2H(\lambda)}}\lambda\right)}=t, \qquad \forall (t,q)\in (0,\varepsilon)\times X_\varepsilon.
\end{equation}
Notice that $F$ is the gradient flow of $\delta$ on $M_\varepsilon$.
Now, for $q\in X_\varepsilon$, the curves $t\mapsto F(t ,q)$ are the unique normal geodesics with speed $1$ that minimize the sub-Riemannian distance from $\mathcal Z$.
Hence, $F_*\partial_t$ is a horizontal vector field and $|F_*\partial_t|_{}=1$.
We conclude the proof by showing that that $\nabla \delta = F_*\partial_t$. In fact, by Cauchy-Schwarz inequality, if $\nabla \delta$ is not parallel to $F_*\partial_t$, then $1=|g(F_*\partial_t,\nabla \delta)| < |\nabla \delta|_{}$ at some point $F(\bar t,\bar q)$. On the other hand, the unit-speed curve $\gamma(s) = e^{s \nabla \delta/|\nabla \delta|_{}}F(\bar t,\bar q)$ satisfies, for $T$ small enough,
\begin{equation}
\begin{split}
\delta (\gamma(T)) - \delta (\gamma(0)) &
= \int_0^T \dfrac{d}{ds}\delta(\gamma(s))\Big|_{s=t}dt
= \int_0^T g(\nabla\delta(\gamma(t)),\dot \gamma(t))dt\\
&= \int_0^T g\left(\nabla \delta,\frac{\nabla \delta}{|\nabla \delta|_{}}\right)dt 
= \int_0^T |\nabla \delta |_{} >  T = \ell(\gamma|_{[0,T]}),
\end{split}
\end{equation}
leading to a contradiction, and implying the statement.
\end{proof}

\begin{rmk}
In Proposition \ref{prop:smoothness}, one can replace the smoothness of $\mathcal Z$ by its $C^k$-regularity, $k\geq 2$, obtaining in (ii) the $C^k$-regularity of $\delta$, and in (iii) the $C^{k-1}$-regularity of $F$. Moreover, taking into account the observation of the footnote above, the argument of the proof can be adapted to yield a generalization of the co-dimension $1$ case in \cite[Thm.~4.2]{R17}. This can be done by exploiting the Ball-Box Theorem \cite[Thm.~10.62]{nostrolibro} to estimate $\operatorname{Reach}(S,K)$ of \cite[Thm.~4.2]{R17} in terms of the $\varepsilon$ given by Proposition \ref{prop:smoothness} applied to $\mathcal Z=\{p\in S : d(p,K)<r\}$, for sufficiently small $r>0$.
\end{rmk}

\section{Main quantum completeness criterion}
\label{sec:main-criterion}
Let $N$ be a complete sub-Riemannian manifold and $\mathcal{Z} \subset N$ be a smooth embedded hypersurface with no characteristic points. Let $\omega$ be a measure on $N$, smooth on $M=N\setminus \mathcal Z$ or one of its connected components.
We are interested in the essential self-adjointness of the operator
\begin{equation}
H=-\Delta_\omega=-\dive_\omega \circ \nabla, \qquad \dom(H) = C_c^\infty(M).
\end{equation}

In the following, we denote with $L^2(M)$ the complex Hilbert space of (equivalence classes of) functions $u : M \to \mathbb{C}$, with scalar product
\begin{equation}
\label{eq:L2}
\langle u, v \rangle  = \int_M u \bar{v}\, d \omega, \qquad u,v \in L^2(M),
\end{equation}
where the bar denotes complex conjugation. The corresponding norm is $\|u\|^2 = \langle u,u \rangle$.
Similarly, given a coordinate neighborhood $U\subseteq M$ and denoting by $dx$ the Lebesgue measure on it, we denote by $L^2(U,dx)$ the complex Hilbert space of square-integrable functions $u : U\! \to\! \mathbb{C}$ satisfying \eqref{eq:L2} with $d\omega$ replaced by $dx$ and $M$ by $U$.

Our main result is the following.

\begin{theorem}[Main quantum completeness criterion]\label{thm:main}
Let $N$ be a complete sub\--\-Rie\-man\-nian manifold endowed with a measure $\omega$. Assume $\omega$ to be smooth on $N\setminus \mathcal{Z}$, where the singular set $\mathcal Z$ is a smooth, embedded, compact hypersurface with no characteristic points.
Assume also that, for some $\varepsilon>0$, there exists a constant $\kappa\geq 0$ such that, letting $\delta = d(\mathcal{Z}, \cdot\,)$, we have
\begin{equation}\label{eq:cond}
\Veff = \left(\frac{\Delta_\omega\delta}{2}\right)^2 + \left(\frac{\Delta_\omega\delta}{2}\right)^\prime \geq \frac{3}{4\delta^2} -\frac{\kappa}{\delta}, \qquad \text{for } 0<\delta \leq \varepsilon,
\end{equation}
where the prime denotes the derivative in the direction of $\nabla\delta$. 
Then, $\Delta_\omega$ with domain $C^\infty_c(M)$ is essentially self-adjoint in $L^2(M)$, where $M = N \setminus \mathcal{Z}$, or any of its connected components.

Moreover, if $M$ is relatively compact, the unique self-adjoint extension of $\Delta_\omega$ has compact resolvent. Therefore, its spectrum is discrete and consists of eigenvalues with finite multiplicity.
\end{theorem}
\begin{rmk}\label{rmk:cpt}
The compactness of $\mathcal{Z}$ in Theorem~\ref{thm:main} can be replaced by the weaker assumption that the (normal) injectivity radius from $\mathcal{Z}$ is strictly positive. Indeed, in this case, Proposition~\ref{prop:smoothness} and the forthcoming Proposition~\ref{prop:weak-hardy} still hold true. (See footnote in the proof of Proposition~\ref{prop:smoothness}.)
\end{rmk}

We start by showing two functional theoretic results holding on any sub-Riemannian manifold $M$ equipped with a smooth measure $\omega$.

We denote by $W^1(M)$ the Sobolev space of functions in $L^2(M)$ with distributional (sub-Riemannian) gradient $\nabla u \in L^2(\mathcal D)$, where the latter is the complex Hilbert space of sections of the complexified distribution $X : M \to \mathcal D^\mathbb{C}\subseteq TM^\mathbb{C}$, with scalar product
\begin{equation}
\label{eq:L2tang}
\langle X, Y \rangle  = \int_M g(X,Y)\, d \omega, \qquad  X,Y \in L^2(\mathcal D).
\end{equation}
The Sobolev space $W^1(M)$ is a Hilbert space when endowed with the scalar product
\begin{equation}\label{eq:W1}
\langle u, v \rangle_{W^1}= \langle \nabla u, \nabla v \rangle + \langle u, v \rangle.
\end{equation}
Similarly, given a coordinate neighborhood $U\subseteq M$ and denoting by $dx$ the Lebesgue measure on it, we denote by $W^1(U,dx)$ the Sobolev space of functions in $L^2(U,dx)$, with distributional (sub-Riemannian) gradient in $L^2(\mathcal D|_U,dx)$, that is the complex Hilbert space of sections of the complexified distribution $X : U \to \mathcal{D}^{\mathbb C}\subseteq TM^\mathbb{C}$, with the scalar product defined in \eqref{eq:L2tang} where $d\omega$ is replaced by $dx$.
Moreover,
we denote by $L^2_\loc(M)$ and $W^1_\loc(M)$ the space of functions $u: M \to \mathbb{C}$ such that, for any relatively compact domain $\Omega \subseteq M$, their restriction to $\Omega$ belongs to $L^2(\Omega)$ and $W^1(\Omega)$, respectively.  Finally, we let $W^1_0(M)$ be the closure of $C^\infty_c(M)$ w.r.t.\ the norm given in \eqref{eq:W1}.

\begin{lemma}\label{l:shubin}
Let $M$ be a sub-Riemannian manifold equipped with a smooth measure $\omega$. Then $\dom(H^*)\subseteq W^1_\loc(M)$.
\end{lemma}

\begin{proof}
If $u \in \dom(H^*)$, then $\Delta_\omega u \in L^2_\loc(M)$. Let $f=\Delta_\omega u$ (in the weak sense) and let $U$ be a relatively compact coordinate domain of $M$. Then, $f\in L^2(U)$, and, since $\omega$ is a smooth measure on $M$, $f\in L^2(U, dx)$, where $dx$ denotes the Lebesgue measure on $U$. Notice that $\Delta_\omega$ can be written in the form $\mathcal L=\sum_{i=1}^rX_i^2+X_0$ where $X_1,\dots, X_r$ is a local generating family and  $X_0$ is a horizontal vector field. 
Then, by Rotschild and Stein subellipticity theory for $\mathcal L$ (see \cite[Thm.\ 18.d]{RS76}), $u\in W^{1}_\mathrm{loc}(U,dx)$, implying $u\in W^{1}_\mathrm{loc}(U)$.
We deduce that $u\in W^1_\mathrm{loc}(M)$. In fact, if $K\subseteq M$ is a relatively compact domain, we can cover it with a finite number of coordinate charts $U_1,\dots, U_m$, with $K\cap U_i$ relatively compact. In particular, $u\in W^1(K\cap U_i)$ for any $i=1,\dots,m$, implying $u\in W^1_\mathrm{loc}(M)$.
\end{proof}

\begin{lemma}[Sub-Riemannian Rellich-Kondrachov theorem]
\label{l:RK}
Let $M$ be a sub-Rie\-mann\-ian manifold equipped with a smooth measure $\omega$. Let $\Omega \subseteq M$ be a compact domain with Lipschitz boundary. Then $W^1(\Omega)$ is compactly embedded into $L^2(\Omega)$. 
\end{lemma}
\proof {\em Step 1.}
Let $U\subseteq M$ be a coordinate neighborhood such that $U\cap \Omega$ has Lipschitz boundary and let $w_j$ be a sequence bounded in $ W^1(\Omega)$.
Since $\omega$ is smooth, this is  equivalent to say that $w_j$ is bounded in $W^1(U\cap \Omega,dx)$, where $dx$ denotes the Lebesgue measure on $U$. By \cite[Thm 13]{RS76} (and estimates therein), if $s$ denotes the step of the sub-Riemannian structure on $\Omega$,  $W^1(U\cap \Omega,dx)$ is compactly embedded into the isotropic fractional Sobolev space $W^{1/s,2}_{\mathrm{iso}}(U\cap \Omega,dx)$. This is defined considering fractional derivatives in every coordinate direction (and not just in the horizontal ones). Then, by the classical Rellich-Kondrachov theorem applied to set $U\cap \Omega$, whose boundary is Lipschitz, we can extract a subsequence $w_{j_\ell}$ of $w_j$ converging in $L^2(U\cap\Omega,dx)$, hence in $L^2(\Omega\cap U)$.

{\em Step 2.} Let $u_j$ be a sequence bounded in $W^1(\Omega)$ and let $\Omega=\bigcup_{\ell=1}^NU_\ell$ be a covering of $\Omega$ where each $U_\ell$ is a coordinate domain. Then $u_j$ is bounded in $W^{1}(U_\ell)$ for every $\ell$. Without loss of generality we can assume $U_\ell\cap \Omega$ to have Lipschitz boundary for every $\ell$. 
By Step 1, we can extract from $u_{j}$ a subsequence $u_{j_1(k)}$ converging in $L^2(\Omega\cap U_1)$. Similarly, from $u_{j_1(k)}$ we extract a subsequence $u_{j_2(k)}$ converging in $L^2(\Omega\cap U_2)$. By repeating this procedure for every $\ell$ we obtain a subsequence $u_{j_N(k)}$ of $u_j$ converging in $L^2(\Omega\cap U_\ell)$ for every $\ell=1,\dots, N$. This implies that $u_{j_N(k)}$ converges in $L^2(\Omega)$, as claimed.
\endproof

\subsection{Agmon-type estimates and weak Hardy inequality}

Recall that the symmetric bilinear form associated with $H$ is
\begin{equation}\label{eq:dirichletform}
	\mathcal{E}(u,v)=\int_M g(\nabla u, \nabla v)  \,d\omega, \qquad u,v \in C^\infty_c(M).
\end{equation}
We use the same symbol to denote the above integral for all functions $u, v \in W^1_\loc(M)$, when it is convergent. We also let, for brevity, $\mathcal{E}(u) = \mathcal{E}(u,u) \geq 0$. 

\begin{lemma}\label{l:trick}
Let $M$ be a sub-Riemannian manifold equipped with a smooth measure $\omega$.
Let $f$ be a real-valued function, Lipschitz w.r.t.\ the sub-Riemannian distance. Let $u \in W^1_\loc(M)$, and assume that $f$ or $u$ have compact support $K \subset M$. Then, we have
\begin{equation}\label{eq:keyidentity}
\mathcal{E}(f u,fu) = \re \mathcal{E}(u,f^2 u ) + \langle u , |\nabla f|^2 u \rangle.
\end{equation}

Moreover, if $\psi \in \dom(H^*)$ satisfies $H^* \psi = E \psi$, and $f$ has compact support, we have
\begin{equation}\label{eq:keyidentity2}
\mathcal{E}(f \psi,f\psi) = E \| f \psi \|^2 + \langle \psi , |\nabla f|^2 \psi \rangle.
\end{equation}
\end{lemma}
\begin{proof}
Observe that $|\nabla f|$ is essentially bounded by \cite[Thm.\ 8]{FHK99} (see also \cite[Prop.\ 2.9]{FSSC97}, \cite[Thm.\ 1.3]{GN98}). Hence $fu \in W^1_\comp(M)$. By using the fact that $f$ is real-valued, a straightforward application of Leibniz rule yields
\begin{align}
\langle  \nabla u, \nabla (f^2 u)\rangle & = \langle  f\nabla u, \nabla (f u)\rangle  + \langle \nabla u, f u \nabla f  \rangle \\
& = \langle \nabla (fu), \nabla (fu)\rangle - \langle u \nabla f,\nabla (fu) \rangle+ \langle \nabla u, f u \nabla f  \rangle \\
& = \langle \nabla (fu), \nabla (fu)\rangle - \langle u \nabla f,u\nabla f \rangle- \langle u \nabla f, f\nabla u \rangle +\langle f \nabla u,  u \nabla f  \rangle  \\
& = \langle \nabla (fu), \nabla (fu)\rangle - \langle u, |\nabla f|^2 u\rangle + 2 i \im \langle f \nabla u,  u \nabla f  \rangle.
\end{align}
Thus, by definition of $\mathcal{E}$, we have
\begin{align}
\re \mathcal{E}(u,f^2 u) & =  \langle \nabla (fu), \nabla (fu)\rangle  - \langle u, |\nabla f|^2 u \rangle = \mathcal{E}(fu, fu) - \langle u, |\nabla f|^2 u \rangle,
\end{align}
completing the proof of \eqref{eq:keyidentity}.

To prove \eqref{eq:keyidentity2}, recall that, by Lemma \ref{l:shubin}, $\dom(H^*) \subseteq W^1_\loc(M)$.
Then we obtain
\begin{align}
\mathcal{E}(u,f^2 u) & =  \langle \nabla u , \nabla (f^2 u) \rangle = \langle-\Delta_\omega u, f^2 u\rangle = \langle H^* u, f^2 u \rangle .
\end{align}
Setting $u= \psi$, we obtain $\mathcal{E}(\psi,f^2 \psi) = E \|f \psi\|^2$, yielding the statement.
\end{proof}

We show how to compute $V_{\mathrm{eff}}$ through the diffeomorphism $F$ given by Proposition \ref{prop:smoothness}.
\begin{prop}\label{prop:meas}
Using the diffeomorphism of Proposition \ref{prop:smoothness} to identify $M_\varepsilon\simeq (0,\varepsilon)\times X_\varepsilon$, we have
\begin{equation}
d\omega(t,q)=e^{2\theta(t,q)}dt\;d\mu(q),\qquad (t,q)\in M_\varepsilon,
\end{equation}
where $d\mu$ is a fixed smooth measure on $X_\varepsilon$, and $\theta$ is a smooth function. Moreover,
\begin{equation}
\label{eq:Veff}
\Veff=(\partial_t\theta)^2+\partial_t^2\theta.
\end{equation}
\end{prop}
\begin{proof}
We prove \eqref{eq:Veff}. Through the identification $M_\varepsilon\simeq (0,\varepsilon)\times X_\varepsilon$ we have $\nabla \delta(t,q)=\partial_t$. Then, by definition of $\dive_{\omega}$ we have
\begin{equation}
\begin{split}
(\Delta_\omega\delta(t,q))\omega&=\dive_\omega(\partial_t)\omega=\mathcal L_{\partial_t}\omega=\mathcal L_{\partial_t}(e^{2\theta(t,q)}dt\;d\mu(q))\\
&=2\partial_t\theta(t,q)d\omega+e^{2\theta}\mathcal L_{\partial_t}(dt\;d\mu(q))=
2\partial_t\theta(t,q)d\omega,
\end{split}
\end{equation}
where we used $\mathcal L_{\partial_t}(dt\,d\mu(q))=0$. Moreover, in these coordinates, derivation in the direction of $\nabla\delta$ amounts to the derivation w.r.t.\ $t$, hence
\begin{equation}
\left(\Delta_\omega\delta(t,q)\right)'=2\partial_t^2\theta. \qedhere
\end{equation}
\end{proof}

\begin{prop}[Weak Hardy Inequality]\label{prop:weak-hardy}
Let $N$ be a complete sub-Rie\-mann\-ian manifold endowed with a measure $\omega$. Assume $\omega$ to be smooth on $M=N\setminus \mathcal{Z}$, where the singular set $\mathcal Z$ is a smooth, embedded, compact hypersurface with no characteristic points.
Assume also that there exist $\kappa \geq 0$ and $\varepsilon>0$ such that
	\begin{align}\label{eq:assumptionprop}
\Veff & \geq \frac{3}{4\delta^2}-\frac{\kappa}{\delta},\qquad \text{for $\delta \leq \varepsilon$}.
	\end{align}
Then, there exist $\eta \leq 1/\kappa$ and $c \in \R$ such that
	\begin{equation}\label{eq:weak-hardy}
\int_M |\nabla u|^2\;d\omega  \geq	\int_{M_{\eta}} \left(\frac{1}{\delta^2} - \frac{\kappa}{\delta} \right)|u|^2\,d\omega  +c \|u\|^2, \qquad \forall u\in W^1_\comp(M),
	\end{equation}
where $M_\eta=\{ 0<\delta<\eta\}$. In particular, the operator $H=-\Delta_\omega$ is semibounded on $C^\infty_c(M)$.
\end{prop}
\begin{proof} By Proposition \ref{prop:smoothness} there exists $\varepsilon >0$ such that $\delta$ is smooth on $M_\varepsilon=\{0<\delta<\varepsilon\}$.

First we prove \eqref{eq:weak-hardy} for $u \in W^{1}_{\comp}(M_{\varepsilon})$, and with $\eta = \varepsilon$, possibly not satisfying $\eta \leq 1/\kappa$. Then, we extend it for $u \in W^1_\comp(M)$, choosing $\eta\le 1/\kappa$. 

\emph{Step 1.} Let $u \in W^{1}_{\comp}(M_{\varepsilon})$. By Proposition~\ref{prop:smoothness}, we identify $M_\varepsilon \simeq (0,\varepsilon)\times \mathcal Z$ in such a way that $\delta(t,q) = t$. By Proposition~\ref{prop:meas}, fixing a reference measure $d\mu$ on $\mathcal Z$, we have
$d\omega (t,q)= e^{2 \vartheta(t,q)} dt\, d\mu(q)$ on  $M_\varepsilon$,
for some smooth function $\vartheta : M_\varepsilon \to \R$.
Consider the unitary transformation $T :L^2(M_\varepsilon,d\omega)\to L^2(M_\varepsilon,dt\,d\mu )$ defined by $Tu = e^\vartheta u$. By Proposition \ref{prop:smoothness} $\partial_t$ is a unit horizontal vector field. Hence $|\nabla u|\geq |\partial_tu|$. Letting $v=T u$, an integration by parts yields 	
\begin{equation}
			\int_M |\nabla u|^2\;d\omega 
			\geq \int_{M_\varepsilon} |\partial_t u|^2 \,d\omega 
			=  \int_{M_\varepsilon} \bigg(|\partial_t v|^2 + \bigg(\underbrace{(\partial_t \vartheta)^2 + \partial^2_t \vartheta}_{= \Veff} \bigg) |v|^2\bigg) \, dt\,d\mu,
	\end{equation}
where the expression for $\Veff$ is in Proposition~\ref{prop:meas}. Recall the 1D Hardy inequality:
\begin{equation}\label{eq:Hardy1D}
\int_0^\varepsilon |f'(s)|^2\, ds  \geq \frac{1}{4}\int_0^\varepsilon \frac{|f(s)|^2}{s^2}\,ds, \qquad \forall f \in W^{1}_\comp((0,\varepsilon)).
\end{equation}
Since $u\in W^1_\comp(M_\varepsilon)$ and $\vartheta$ is smooth, for a.e.\ $q \in X_\varepsilon$, the function $t \mapsto v(t,q)$ is in $W^{1}_\comp((0,\varepsilon))$ (see \cite[Thm.\ 4.21]{EGmeas}). Then, by using \eqref{eq:assumptionprop}, Fubini's Theorem and \eqref{eq:Hardy1D}, we obtain \eqref{eq:weak-hardy} for functions $u \in W^{1}_{\comp}(M_{\varepsilon})$ with $\eta = \varepsilon$ and $c=0$.
	
\emph{Step 2.} 	Let $u \in W^1_\comp(M)$, and let $\chi_1,\chi_2$ be smooth functions on $[0,+\infty)$ such that
\begin{itemize}
\item $0 \leq \chi_i \leq 1$ for $i=1,2$;
\item $\chi_1 \equiv 1$ on $[0,\frac{\varepsilon}{2}]$ and $\chi_1 \equiv 0$ on $[\varepsilon,+\infty)$;
\item $\chi_2 \equiv 0$ on $[0,\frac{\varepsilon}{2}]$ and $\chi_2 \equiv 1$ on $[\varepsilon,+\infty)$;
\item $\chi_1^2+\chi_2^2 =1 $.
\end{itemize}
Consider the functions $\phi_i : M \to \R$ defined by $\phi_i:= \chi_i \circ \delta$. We have $\phi_1 \equiv 1$ on $M_{\varepsilon/2}$, $M_{\varepsilon/2} \subseteq \mathrm{supp}(\phi_1) \subseteq M_\varepsilon$, moreover $0\leq \phi_1 \leq 1$, and $\phi_1^2 + \phi_2^2 = 1$. Notice that $\phi_2 \equiv 1$ and $\phi_1 \equiv 0 $ on $M \setminus M_\varepsilon$, and so $\nabla \phi_i \equiv 0$ there. Moreover, since by Proposition \ref{prop:smoothness} i) there holds $|\nabla\delta|\leq 1$ a.e., we have
\begin{equation}
\label{eq:est}
c_1=  \sup_{M} \sum_{i=1}^2|\nabla \phi_i|^2 \leq \sup_{[0,\varepsilon]} \sum_{i=1}^2 |\chi_i'|^2< +\infty.
\end{equation}
By \eqref{eq:keyidentity} of Lemma~\ref{l:trick}, we obtain 
\begin{align}
\mathcal{E}(u) = \sum_{i=1}^2 \mathcal{E}(\phi_i u) -\sum_{i=1}^2  \int_M |\nabla \phi_i|^2 |u|^2 d\omega 
 \geq \mathcal{E}(\phi_1 u) -  c_1 \|u\|^2,
\end{align}
where we used that $\mathcal{E}(\phi_2 u) \geq 0$, that $\phi_1^2+\phi_1^2=1$ and the inequality \eqref{eq:est}. 
In particular, applying the statement proved in Step 1 to $\phi_1 u \in W^1_\comp(M_\varepsilon)$, we get
\begin{align}
\mathcal{E}(u)  &\geq \int_{M_\varepsilon}\left(\frac{1}{\delta^2}-\frac{\kappa}{\delta}\right)|\phi_1 u|^2 d\omega - c_1\|u\|^2. 
\end{align}
Letting $\eta = \min\{\frac{\varepsilon}{2},1/\kappa\}$, we have
\begin{align}
	\mathcal E(u)& \geq \int_{M_{\eta}} \left(\frac{1}{\delta^2}-\frac{\kappa}{\delta}\right)|u|^2 d\omega- \int_{M_{\varepsilon}\setminus M_{\eta}} \left\lvert\frac{1}{\delta^2}-\frac{\kappa}{\delta}\right\rvert |\phi_1 u|^2 d\omega -c_1\|u\|^2 \\
& \geq \int_{M_{\eta}} \left(\frac{1}{\delta^2}-\frac{\kappa}{\delta}\right)|u|^2 d\omega - \left(c_1 + \sup_{\eta \leq \delta \leq \varepsilon}\left\lvert \frac{1}{\delta^2} - \frac{\kappa}{\delta}\right\rvert \right) \|u\|^2,
\end{align}
which concludes the proof.
\end{proof}

\begin{prop}[Agmon-type estimate]\label{p:Agmon}
Let $N$ be a complete sub-Rie\-mann\-ian manifold endowed with a measure $\omega$. Assume $\omega$ to be smooth on 
$M=N\setminus \mathcal{Z}$, where the singular set $\mathcal Z$ is a smooth embedded  hypersurface with no characteristic points.
Assume also that there exist $\kappa \geq 0$, $\eta \leq 1/\kappa$ and $c \in \R$ such that,
	\begin{equation}\label{eq:agmon}
\int_M|\nabla u|^2\;d\omega \geq  \int_{M_\eta} \left(\frac {1} {\delta^2} - \frac \kappa \delta\right)|u|^2 d\omega + c \|u\|^2 , \qquad \forall u \in W^{1}_\comp(M).
	\end{equation}
Then, for all $E < c$, the only solution of $\op^* \psi = E \psi$ is $\psi \equiv 0$.
\end{prop}
Notice that the requirement $\eta \leq 1/\kappa$ ensures the non-negativity of the integrand in \eqref{eq:agmon}. The proof follows the ideas of \cite{Nenciu2008,DeVerdiere2009a}.

\begin{proof}
Let $f: M \to \R$ be a bounded Lipschitz function w.r.t.\ the sub-Riemannian distance with $\supp f \subseteq \overline{M \setminus M_\zeta}$, for some $\zeta >0$, and $\psi$ be a solution of $(H^* - E) \psi = 0$ for some $E <c $. We start by claiming that
\begin{equation}\label{eq:agmondfirststep}
(c-E) \|f \psi\|^2 \leq \langle \psi, |\nabla f |^2 \psi \rangle - \int_{M_\eta} \left(\frac{1}{\delta^2} - \frac{\kappa}{\delta}\right) |f \psi|^2 d\omega.
\end{equation}
If $f$ had compact support, then $f \psi \in W^1_\comp(M)$, and hence \eqref{eq:agmondfirststep} would follow directly from \eqref{eq:agmon} and \eqref{eq:keyidentity2}. To prove the general case, let $\theta : \R \to \R$ be the function defined by
\begin{equation}
\theta(s) = \begin{cases}
1 & s \leq 0, \\
1-s & 0\leq s \leq 1, \\
0 & s \geq 1.
\end{cases}
\end{equation}
Fix $q \in M$ and let $G_n : M \to \R$ defined by $G_n(p) = \theta(d_g(q,p) - n )$. Notice that $G_n$ is Lipschitz w.r.t.\ the sub-Riemannian distance, and hence its sub-Riemannian gradient satisfies $|\nabla G_n | \leq 1$, see \cite[Thm.\ 8]{FHK99}, \cite[Prop.\ 2.9]{FSSC97}, \cite[Thm.\ 1.3]{GN98}. Moreover $\supp(G_n) \subseteq \bar{B}_q(n+1)$. Observe that
\begin{equation}\label{eq:compactnesssupport}
\supp G_n f \subseteq \overline{(M\setminus M_\zeta) \cap B_q(n+1)}.
\end{equation}
Even if $(M,d)$ is a non-complete metric space (and hence, its closed balls might fail to be compact), the set on the right hand side of \eqref{eq:compactnesssupport} is compact, being uniformly separated from the metric boundary. This can be proved with the same argument of \cite[Prop.\ 2.5.22]{BBI} 
and exploiting the completeness of $(N,d)$. Hence, the support of $f_n:= G_n f$ is compact, and \eqref{eq:agmondfirststep} holds with $f_n$ in place of $f$.
The claim now follows by dominated convergence. Indeed, $f_n \to f$ point-wise as $n \to +\infty$ and $f_n \leq f$. Hence $\|f_n \psi \| \to \|f \psi \|$. Thus, since $\supp f_n \subseteq \overline{M \setminus M_\zeta}$, we have
\begin{equation}
\lim_{n \to +\infty} \int_{M_\eta} \left(\frac{1}{\delta^2} - \frac{\kappa}{\delta}\right) |f_n \psi|^2\, d\omega= \int_{M_\eta} \left(\frac{1}{\delta^2} - \frac{\kappa}{\delta}\right) |f \psi|^2\, d\omega.
\end{equation}
Finally, since $|\nabla f_n | \leq C$, and $\nabla f_n \to \nabla f$ a.e.\ we have $\langle \psi, |\nabla f_n |^2 \psi \rangle \to \langle \psi, |\nabla f |^2 \psi \rangle$, yielding the claim.

We now plug a particular choice of $f$ into \eqref{eq:agmondfirststep}. Set
\begin{equation}\label{eq:functin}
f(p) := \begin{cases}
F(\delta(p)) &  0 < \delta(p) \leq \eta, \\
1 & \delta(p) > \eta ,
\end{cases}
\end{equation}
where $F$ is a Lipschitz function to be chosen later. Recall that $|\nabla\delta| \leq 1$ a.e. on $M$. In particular, a.e.\ on $M_\eta$, we have $|\nabla f| = |F'(\delta)| |\nabla\delta | \leq |F'( \delta)|$. Thus, by \eqref{eq:agmondfirststep}, we have
\begin{equation}\label{eq:agmondsecondstep}
(c-E) \|f \psi\|^2 \leq  \int_{M_\eta} \left[F'(\delta)^2 -\left(\frac{1}{\delta^2} - \frac{\kappa}{\delta}\right) F(\delta)^2 \right]|\psi|^2 d\omega .
\end{equation}
Let now $0<2\zeta <\eta$. We choose $F$ for $\tau \in [2\zeta,\eta]$ to be the solution of
\begin{equation}\label{eq:costruzioneFAgmon}
F'(\tau) = \sqrt{\frac{1}{\tau^2}- \frac{\kappa}{\tau}} F(\tau), \qquad \text{with } F(\eta) = 1,
\end{equation}
to be zero on $[0,\zeta]$, and linear on $[\zeta,2\zeta]$. 
Observe that the assumption $\eta \leq 1/\kappa$ implies that \eqref{eq:costruzioneFAgmon} is well defined. 
We first consider the case $\kappa=0$. The function $F$, together with its derivative reads
\begin{equation}
F(\tau)=\begin{cases}
0&\tau\in[0,\zeta],\\
\frac{2}{\eta}(\tau-\zeta)&\tau\in[\zeta,2\zeta],\\
\frac{1}{\eta}\tau&\tau\in[2\zeta,\eta),
\end{cases}
\qquad
F'(\tau)=\begin{cases}
0&\tau\in[0,\zeta],\\
\frac{2}{\eta}&\tau\in[\zeta,2\zeta],\\
\frac{1}{\eta}&\tau\in[2\zeta,\eta).
\end{cases}
\end{equation}
The global function defined by \eqref{eq:functin} is a Lipschitz function with support contained in $\overline{M\setminus M_\zeta}$ and such that $F' \leq K$ on $[\zeta,2\zeta]$, for some constant independent of $\zeta$ ($K=2/\eta$). Therefore, from \eqref{eq:agmondsecondstep} we get
\begin{equation}\label{eq:laststepk0}
(c-E)\|f\psi\|^2
\leq \int_{M_{2\zeta}\setminus M_{\zeta}} \left[F'(\delta)^2 -\frac{1}{\delta^2}  F(\delta)^2 \right]|\psi|^2 d\omega
\leq K^2\int_{M_{2\zeta}\setminus M_\zeta}|\psi|^2d\omega.
\end{equation}
If we let $\zeta \to 0$, then $f$ tends to an almost everywhere strictly positive function. Recalling that $E<c$, and taking the limit, equation \eqref{eq:laststepk0} implies $ \psi \equiv 0$.
When $\kappa>0$  the solution to \eqref{eq:costruzioneFAgmon}, on the interval $[2\zeta,\eta]$, is
\begin{equation}
F(\tau)=C(\kappa,\eta) \frac{1-\sqrt{1-\kappa  \tau}}{1+\sqrt{1-\kappa  \tau}}e^{2 \sqrt{1-\kappa  \tau}}, \qquad \tau \in[2\zeta,\eta],
\end{equation}
for a constant $C(\kappa,\eta)$ such that $F(\eta) = 1$. By construction of $F$ on $[\zeta,2\zeta]$, we obtain
\begin{equation}
F'(\tau) = \frac{F(2\zeta)}{\zeta}, \qquad \tau \in [\zeta,2\zeta].
\end{equation}
Hence we have $F(2\zeta) = C(\kappa,\eta) e^2 \kappa \zeta/2 + o(\zeta)$, which yields the boundedness of $F'$ on $[\zeta,2\zeta]$ by a constant not depending on $\zeta$. Moreover the global function defined by \eqref{eq:functin} is Lipschitz with support contained in $\overline{M\setminus M_\zeta}$. Thus, by \eqref{eq:agmondsecondstep}, we conclude that $\|\psi\|=0$ as in the case $\kappa=0$.
\end{proof}

\begin{rmk}[The role of the Hardy constant in the proof] If \eqref{eq:agmon} is replaced with 
\begin{equation}
\int_M|\nabla u|^2\;d\omega \geq  a\int_{M_\eta} \left(\frac {1} {\delta^2} - \frac \kappa \delta\right)|u|^2 d\omega + c \|u\|^2 , \qquad \forall u \in W^{1}_\comp(M)
\end{equation}
for $\frac{3}{4}<a<1$, then the arguments in the previous proof cannot be applied.
To see this, let us consider the case $\kappa=0$. The function $F$ satisfying a suitably modified version of \eqref{eq:costruzioneFAgmon} reads in this case 
\begin{equation}
F(\tau)=\left(\frac{\tau}{\eta}\right)^{\sqrt{a}},\qquad \tau\in [2\zeta,\eta].
\end{equation} 
Then, by construction, the function $F$ satisfies 
\begin{equation}
F(\tau)=\left(\frac{2}{\eta}\right)^{\sqrt{a}}\zeta^{\sqrt{a}-1}(\tau-\zeta)\quad \text{for } [\zeta,2\zeta].
\end{equation}
In particular, if $a<1$, we cannot find a constant $K$ independent of $\zeta$ such that $F'(\tau)=({2}/{\eta})^{\sqrt{a}}\zeta^{\sqrt{a}-1}\leq K$.
On the other hand,  for $a\geq 1$, we have $F'(\tau)=({2}/{\eta})^{\sqrt{a}}\zeta^{\sqrt{a}-1}\leq 2^{\sqrt{a}}/\eta^{2-\sqrt{a}}=K(\eta)$ and the previous argument works exactly in the same way.
\end{rmk}

\subsection{Proof of the criterion}



\begin{proof}[Proof of Theorem~\ref{thm:main}]
We divide the proof of the theorem in two steps.

\vspace{.2cm}
\textit{Part 1: essential self-adjointness.}
By Proposition~\ref{prop:weak-hardy}, the operator $\op$ is semibounded. Thus, by a well-known criterion (see \cite[Thm.\ X.I and Corollary]{MR0493420}), $\op$ is essentially self-adjoint if and only if there exists $E <0$ such that the only solution of $\op^*\psi = E\psi$ is $\psi \equiv 0$. This is guaranteed by the Agmon-type estimate of Proposition \ref{p:Agmon}, whose hypotheses are satisfied again by Proposition \ref{prop:weak-hardy}.

\vspace{.2cm}
\textit{Part 2: compactness of the resolvent.} The proof follows the same steps as in \cite[Prop.\ 3.7]{quantum-confinement}, but makes use of the sub-Riemannian version of the Rellich-Kondrachov theorem (Lemma \ref{l:RK}). For the sake of completeness, we sketch here the proof.

First of all notice that it suffices to show that there exists $z\in\R$ such that the resolvent $(H^*-z)^{-1}$ is compact on $L^2(M)$. This follows by the first resolvent formula (see \cite[Thm.\ VIII.2]{MR0493419}) and by the fact that compact operators are an ideal of the algebra of bounded ones. Moreover, by Proposition \ref{prop:weak-hardy}, $H$ is a semibounded operator, i.e.,
\begin{equation}\label{eq:sbdd}
\langle Hu, u\rangle\geq c\|u\|^2, \qquad \forall u\in C^\infty_c(M).
\end{equation} 
Hence, by \cite[Thm.\ XIII.64]{MR0493421} its spectrum consists of discrete eigenvalues with finite multiplicity.

Notice that \eqref{eq:sbdd}, together with the fact that $H^*$ is self-adjoint, imply that $(H^*- z)^{-1}$ is well defined for every $z\leq c$ and 
$\|(H^*-z)^{-1}\|\leq 1/(c-z)$.
To prove compactness of the operator $(H^*-z)^{-1}:L^2(M)\to \dom(H^*)$ for $z<c$ we need to show that for any bounded sequence $\psi_n\in L^2(M)$, say $\|\psi_n\|\leq (c-z)$, the image sequence $u_n=(H^*-z)^{-1}\psi_n\in \dom(H^*)$ has a subsequence converging in $L^2(M)$. Notice that $\|u_n\|\leq 1$.

In order to extract a converging subsequence of $u_n$, we prove estimates for the functions $u_n$ localized close and far away from the singular region. We provide such estimates for any function $u\in \dom(H^*)\subseteq W^1_{\mathrm{loc}}(M)$ (see Lemma \ref{l:shubin}), setting $\psi=(H^*-z)u$, and we will then apply them to the elements of the sequence $u_n$, to extract a converging subsequence. To this purpose let $\chi_1,\chi_2:[0,+\infty]\to \mathbb R$ be real valued Lipschitz functions such that
\begin{itemize}
\item $0 \leq \chi_i \leq 1$ for $i=1,2$;
\item $\chi_1 \equiv 1$ on $[0,\eta/2]$ and $\chi_1 \equiv 0$ on $[\eta,+\infty)$;
\item $\chi_2 \equiv 0$ on $[0,\eta/2]$ and $\chi_2 \equiv 1$ on $[\eta,+\infty)$;
\item they interpolate linearly elsewhere.
\end{itemize}
Consider the functions $\phi_i:= \chi_i \circ \delta$, which are Lipschitz w.r.t.\ the sub-Riemannian distance. Notice that $\phi_1 + \phi_2 = 1$. Since $M$ is relatively compact in $N$, $\phi_2$ is compactly supported in $M$, implying by \eqref{eq:keyidentity}
\begin{align}
\mathcal{E}(\phi_2 u,\phi_2u) & = \re \mathcal{E}(u,\phi_2^2 u ) + \langle u , |\nabla \phi_2|^2 u \rangle=\re \langle H^* u,\phi_2^2u\rangle +\langle u , |\nabla \phi_2|^2 u \rangle\\
& = z \|\phi_2 u\|^2 + \re \langle \psi, \phi^2_2 u\rangle + \langle u,|\nabla \phi_2|^2 u \rangle  \leq z \|u\|^2 + \|\psi\| \|u\| + 4\|u\|^2 \eta^{-2}, \label{eq:boundenergia}
\end{align}
where in the last estimate we used the fact that $\chi_2$ is linear  between $\eta/2$ and $\eta$ hence $\phi_2=\chi_2\circ\delta$ satisfies $|\nabla \phi_2| \leq |\chi_2^\prime| |\nabla \delta| \leq 2/\eta$. We deduce the following estimate ``far away'' from the singular region:
\begin{equation}
\label{eq:est_far}
\int_{M\setminus M_{\eta/2}}|\nabla (\phi_2u)|^2\;d\omega=\mathcal E(\phi_2u,\phi_2u) \leq z \|u\|^2 + \|\psi\| \|u\| + 4\|u\|^2 \eta^{-2}.
\end{equation}
We now consider the localization of $u$ close to the metric boundary. Since $H$ is essentially self-adjoint and $H^*=\bar H$, we can choose a sequence 
$u_k\in C_c^\infty(M)$ such that $u_k$ converges to $u$ in the graph norm of $H^*$, i.e., $\|H^*(u_k-u)\|+\|u_k-u\|\to0$ as $k\to \infty$. We deduce an upper bound for $\phi_1u$ in $M_\eta$ from the following bounds on the elements $u_k$ as follows. First, we use \eqref{eq:weak-hardy} to obtain
\begin{align}
\int_{M_\eta} |u_k|^2 \, d\omega  & = \int_{M_{\eta}} \frac{\delta^2}{1-\delta\kappa} \frac{1-\delta\kappa}{\delta^2}|u_k|^2 \, d\omega \leq  \frac{\eta^2}{1-\eta \kappa} \int_{M_\eta}\left(\frac{1}{\delta^2} - \frac{\kappa}{\delta}\right) |u_k|^2 \, d \omega \\
& \leq  \frac{\eta^2}{1-\eta \kappa} \left( \mathcal{E}(u_k,u_k) - c \|u_k\|^2\right) = \frac{\eta^2}{1-\eta \kappa} \left( \langle H^* u_k,u_k\rangle - c \|u_k\|^2\right).
\end{align}
Then, passing to the limit $k\to \infty$, and recalling that $\phi_1\leq 1$, we get
\begin{align}
\label{eq:vicinoalbordo}
\int_{M_\eta}|\phi_1u|^2\,d\omega&\leq \int_{M_\eta}|u|^2\,d\omega\\
&\leq \frac{\eta^2}{1-\eta \kappa} \left( \langle H^* u,u\rangle - c \|u\|^2\right)=\frac{\eta^2}{1-\eta \kappa}\left((z-c)\|u\|^2+\|\psi\|\|u\|\right).
\end{align}

We apply the latter construction to each element $u_n=(H^*-z)^{-1}\psi_n\in \dom(H^*)$, setting $u_n=u_{n,1}+u_{n,2}$ with $u_{n,i}=\phi_{n,i}u_n$. Recalling that $\|u_n\|\leq 1$, equation \eqref{eq:est_far} applied to $u=u_n$, $\psi=\psi_n$, implies
\begin{equation}
\|u_{n,2}\|^2_{W^1(M)} = \int_{M \setminus M_{\eta/2}} | \nabla u_{n,2}|^2\, d\omega + \|u_{n,2}\|^2  \leq c + 4\eta^{-2} + 1.
\end{equation}
That is, $u_{n,2}$ is bounded in $W^1(M)$. Moreover, by construction, $u_{n,2} \in W^1_\comp(\Omega)$ where $\Omega =\{\delta\geq\eta/2\}\subset M$ is a compact domain with smooth boundary by Proposition \ref{prop:smoothness}. This implies that $u_{n,2}$ converges up to subsequences in $L^2(\Omega)$ (thus in $L^2(M)$) by the sub-Riemannian Rellich-Kondrachov theorem, see Lemma \ref{l:RK}.

On the other hand, \eqref{eq:vicinoalbordo} 
implies that for some constant $C$ independent of $\eta$, we have
\begin{equation}
\|u_{n,1}\|^2 = \int_{M_{\eta}} |u_{n,1}|^2\, d \omega \leq \frac{\eta^2}{1-\eta \kappa} 2(c-z) \leq C \eta^2,
\end{equation}

Since $\eta$ in the Hardy inequality \eqref{eq:weak-hardy} can be arbitrarily small, say $\tilde{\eta}_k^2 = 1/k$, we actually proved that for all $k \in \N$, there is a subsequence $n \mapsto \gamma_k(n)$ such that $u_{\gamma_k(n)} = \sum_{i=1}^2 u_{\gamma_k(n),i}$ with $\|u_{\gamma_k(n),1}\| \leq C/k$ and $u_{\gamma_k(n),2}$ convergent in $L^2(M)$. Exploiting this fact, we extract a Cauchy subsequence of $u_n$, yielding the compactness of $(H^*-z)^{-1}$, and concluding the proof. 
Details on the extraction are in \cite[Prop.\ 3.7]{quantum-confinement}.
\end{proof}

\section{Applications to the intrinsic sub-Laplacian}
\label{sec:examples}
The main interest of our result is in its application to the study of sub-Riemannian manifolds endowed with the intrinsic Popp's measure. More precisely, given a complete sub-Riemannian manifold $N$, we are interested in studying essential self-adjointness of the sub-Laplacian $\Delta=\Delta_\mathcal{P}$, where $\mathcal P$ is the Popp's measure. As discussed in Section \ref{ss:popp}, $\mathcal{P}$ is smooth on the equiregular region of $N$ (the largest open set on which the sub-Riemannian structure is equiregular), and blows up on its complement: the singular region $\mathcal{Z}$. We assume that $\mathcal{Z}$ is a smooth embedded hypersurface with no characteristic points, and that $\mathcal{Z}$ is compact (or, at least, that it has strictly positive injectivity radius, see Remark~\ref{rmk:cpt}). Furthermore, $\dom(\Delta)=C_c^\infty(M)$ with $M=N\setminus \mathcal Z$ or any of its connected components.

We start by considering a  family of structures generalizing the Martinet structure,  which has been presented in the introduction. These are complete sub-Riemannian structures on $\mathbb R^3$, equiregular outside a hypersurface $\mathcal Z\subset\R^3$, on which the distance from $\mathcal Z$ is explicit. Using Theorem \ref{thm:main} (and Remark \ref{rmk:cpt}) we deduce essential self-adjointness of $\Delta=\Delta_\mathcal{P}$ defined on $C^\infty_c(\mathbb R^3\setminus \mathcal Z)$.

\begin{example}[$k$-Martinet distribution]\label{ex:martinet}
Let $k\in\mathbb N$. We consider the sub-Riemannian structure on $\R^3$ defined by the following global generating family of vector fields:
\begin{equation}
X_1=\partial_x,\qquad X_2=\partial_y+x^{2k}\partial_z.
\end{equation}
The singular region is $\mathcal Z=\{x=0\}$ and the distance from $\mathcal Z$ is $\delta(x,y,z)=|x|$.
Using formula \eqref{eq:popp}, the associated Popp's measure turns out to be
\begin{equation}
\label{eq:pm}
\mathcal P=\frac{1}{2\sqrt{2}k|x|^{2k-1}}\;dx\wedge dy \wedge dz.
\end{equation}
The case $k=1$ is the standard Martinet structure considered in the introduction. Notice that the injectivity radius from $\mathcal{Z}$ is infinite, hence even if $\mathcal{Z}$ is not compact we can apply Theorem~\ref{thm:main}. We compute the effective potential $V_\mathrm{eff}$ using \eqref{eq:Veff}. Indeed we have
\begin{equation}
\theta=\theta(x)=\frac{1}{2}\log\frac{1}{2\sqrt{2}k x^{2k-1}},
\end{equation}
and thus, using \eqref{eq:Veff}, we have
\begin{equation}
V_{\mathrm{eff}}(x)=\frac{4k^2-1}{4x^2}\geq \frac{3}{4x^2}, \qquad \forall k \geq 1.
\end{equation}
Hence \eqref{eq:cond} is satisfied, and $\Delta_{\mathcal{P}}$ with domain $C^\infty_c(\R^3 \setminus \mathcal{Z})$ is essentially self-adjoint.
\end{example}

The study of condition \eqref{eq:cond} is a difficult task, because it requires the explicit knowledge of the distance from the singular set. 
In the following we define a class of sub-Riemannian structures, to which Theorem \ref{thm:main} applies, without knowing an explicit expression for $\delta$.
Let $\varpi$ be a reference measure, smooth and positive on the whole $N$ and let $\mathcal P$ denote Popp's measure, smooth on $M=N\setminus Z$. We define the function $\rho:N\to\R$ by setting
\begin{equation}\label{eq:rho}
\rho(p)=\begin{cases}
\left(\frac{d\mathcal P}{d\varpi}\right)^{-1}(p)& \text{if }p\in{N\setminus \mathcal Z}, \\
0&\text{if }p\in\mathcal Z.
\end{cases}
\end{equation}
This is the unique continuous extension to $\mathcal Z$ of the reciprocal of the Radon-Nikodym derivative of $\mathcal P$ with respect to $\varpi$. Notice that $\rho$ is smooth on $N\setminus\mathcal{Z}$.

\begin{definition}\label{def:popp_reg}
We say that a sub-Riemannian manifold $N$ is {\em Popp-regular} if it is equiregular outside a smooth embedded hypersurface $\mathcal{Z}$ containing no characteristic points, and 
there exists $k\in\N$ such that, for all $q\in\mathcal Z$ there exists a neighborhood $\mathcal O$ of $q$ and a smooth submersion $\psi:\mathcal O\to \R$ such that the function $\rho$ defined in \eqref{eq:rho} satisfies $\rho|_\mathcal O=\psi^k$.
\end{definition}

Definition \ref{def:popp_reg} generalizes the notion of {\em regular} almost Riemannian structure given in \cite[Def.\ 7.10]{quantum-confinement}.
Notice that the sub-Riemannian structure in Example \ref{ex:martinet} is Popp-regular.

\begin{prop}
Let $N$ be a complete and Popp-regular sub-Riemannian manifold, with compact singular set $\mathcal Z$. Then, the sub-Laplacian $\Delta_\mathcal P$ with domain $C_c^\infty(M)$ is essentially self-adjoint in $L^2(M)$, where $M=N\setminus \mathcal Z$ or one of its connected components. Moreover, if $M$ is relatively compact, the unique self-adjoint extension of $\Delta_\mathcal P$ has compact resolvent.
\end{prop}
\begin{proof}
We start by noticing that the proof of Claim iii) in Proposition \ref{prop:smoothness} can be modified in such a way that, for any $q\in\mathcal Z$, we construct  local coordinates $(t,x)\in(-\varepsilon,\varepsilon)\times \R^{n-1}$, defined in a neighborhood $\mathcal O\subset N$ of $q$, with respect to which we have
\begin{equation}\label{eq:local_diffeo}
\mathcal Z\cap \mathcal O=\{t=0\},\qquad
\delta(t,x)=t,\qquad \nabla \delta (t,x)=\partial_t.
\end{equation}
In fact, given a coordinate neighborhood $V\subseteq \mathcal Z$ around $q$ we can choose $\lambda:V\to A\mathcal Z$ to be a smooth non-vanishing local section of the annihilator bundle $A\mathcal Z$ defined in \eqref{eq:annihil} with constant Hamiltonian equal to $1/2$. Then, the map $E(x,t\lambda(x))$ is a smooth diffeomorphism satisfying \eqref{eq:local_diffeo}, where $E$ is defined as in \eqref{eq:E}.

Let now $\varpi$ be a smooth measure on $N$ and consider the  function $\rho$ defined in \eqref{eq:rho}. By assumption, we have $\rho=\psi^k$, for a smooth submersion $\psi$. Thus, since in the coordinates just defined we have $\rho(0,x)=0$, we must also have $\partial_t\psi(0,x)\neq 0$. This implies $\rho = t^k\phi(t,x)$ for a smooth never vanishing function. Notice that the expression of $\phi$ depends on the choice of the reference measure $\varpi$, but the fact that $\phi$ never vanishes does not depend on this choice.
We compute the effective potential as
\begin{align}
\Veff|_{\mathcal{O} \setminus \mathcal{Z}}  &=  \left(\frac{\Delta |t|}{2}\right)^2 + \partial_t\left(\frac{\Delta |t|}{2}\right) \\
& = \frac{k(k+2)}{4t^2} + \frac{k^2}{2|t|}\frac{\partial_t \phi(t,x)}{\phi(t,x)} + \frac{k(k+2)}{4}\frac{\partial_t \phi(t,x)^2}{\phi(t,x)^2} - \frac{k}{2}\frac{\partial_t^2 \phi(t,x)}{\phi(t,x)}.
\end{align}
Up to restricting to a smaller, compact subset $\mathcal O'\simeq[-\varepsilon',\varepsilon']\times[-1,1]^{n-1}$, we get the estimate ${\Veff}|_{\mathcal{ O}'\setminus \mathcal Z}\geq 3/(4t^2)-\kappa'/|t|$ for some constant $\kappa'>0$. 
By compactness of $\mathcal Z$, and up to choosing a sufficiently  small
$\varepsilon$, we can cover $M_{\varepsilon}=\{0<\delta<\varepsilon\}$ with a finite number of coordinate neighborhoods $\mathcal{O}'$ and we obtain the global estimate $\Veff\geq 3/(4\delta^2)- \kappa/\delta$ on $M_\varepsilon$. 
We conclude by applying Theorem \ref{thm:main}.
\end{proof}

\begin{rmk}
The compactness of $\mathcal Z$, used to produce uniform lower bounds for the $V_{\mathrm{eff}}$, is not a necessary condition. For instance, the singular regions of Martinet-type structures of Example~\ref{ex:martinet} are not compact. Nonetheless, the $k$-Martinet structures are Popp-regular and, as we have seen, Theorem \ref{thm:main} still yields the essential self-adjointness of $\Delta_\mathcal{P}$.
\end{rmk}

We generalize Example 7.2 in \cite{quantum-confinement}, showing a family of non\--Popp\--regular sub\--Rie\-mann\-ian structures to which Theorem \ref{thm:main} might apply or not.
\begin{example}[non-Popp-regular sub-Riemannian structure]
Consider the sub-Riemannian structure on $\R^4$ given by the following generating family of vector fields:
\begin{equation}
X_1=\partial_1+x_3\partial_4,\qquad X_2=x_1(x_1^{2\ell}+x_2^2)\partial_2,\qquad X_3=\partial_3.
\end{equation}
The singular region is $\mathcal Z=\{x_1=0\}$.
The following set of vector fields is an adapted frame on $\R^4\setminus\mathcal Z$.
\begin{equation}
\underbrace{X_1,\ X_2,\ X_3}_{\mathcal D^1},\qquad \underbrace{X_4=[X_3,X_1]=\partial_4}_{\mathcal D^2/\mathcal D^1}.
\end{equation}
Using formula \eqref{eq:popp}, we have the following expression for Popp's measure
\begin{equation}
\mathcal P=\frac{1}{\sqrt{2}x_1(x_1^{2\ell}+x_2^2)}dx_1\wedge dx_2\wedge dx_3 \wedge dx_4,
\end{equation}
or, equivalently, $\mathcal P=x_1^{a(x)}e^{2\varphi(x)}dx_1\wedge dx_2\wedge dx_3 \wedge dx_4$, where
\begin{equation} \label{eq:a}
a(x)=\begin{cases}
-(2\ell+1) & x_2=0,\\
-1 &  x_2\neq0,\\
\end{cases}
\qquad
\varphi(x)=\begin{cases}
-\frac{1}{2}\log{\sqrt{2}} & x_2=0, \\
-\frac{1}{2}\log\left(\sqrt{2}(x_1^{2\ell}+x_2^2)\right) & x_2\neq 0.\\
\end{cases}
\end{equation}
Noticing that $\delta(x_1,x_2,x_3,x_4)=x_1$, the effective potential reads
\begin{equation}
\Veff=\frac{a(x)(a(x)-2)}{4x_1^2}+R(x),\ \text{with }R(x)=\frac{a(x)} {x_1}\partial_1\varphi(x)+(\partial_1\varphi(x))^2+\partial_1^2\varphi(x).
\end{equation}
We have
\begin{equation}\label{eq:R}
R(x)=\begin{cases}
0 & x_2=0,\\
\frac{\ell x_1^{2\ell-2}}{(x_1^{2\ell+x_2^2})^2}\left[(\ell+2)x_1^{2\ell}+(2-2\ell)x_2^2\right] & x_2\neq0.
\end{cases}
\end{equation}
Combining \eqref{eq:a}-\eqref{eq:R} we deduce that $\Veff= 3/(4x_1^2)+R(x)$ if $x_2\neq 0$, and it is easy to see that
the behavior of $R(x)$ depends on the choice of the parameter $\ell$. In particular, if $\ell=1$, $R(x)\geq 0$ and we deduce essential self-adjointness of $\Delta=\Delta_{\mathcal P}$ by Theorem \ref{thm:main}. On the other hand, if $\ell>1$, along any sequence $x^i=(1/i,1/i,0,0)$, we have $x_1^i R(x_i) \to -\infty$. Hence, we cannot apply Theorem \ref{thm:main}. 
\end{example}

We show an example of a non-Popp regular sub-Riemannian structure, to which Theorem \ref{thm:main} applies only on one connected component of $N\setminus \mathcal Z$. Indeed, the sub-Laplacian $\Delta_{\mathcal{P}}$ is essentially self-adjoint on one connected component and not on the other one.

\begin{example}
Let $f:\R\to\R$ be the function
\begin{equation}
f(t)=\begin{cases}
\sqrt{2}e^{-\frac{1}{t^2}}&\text{for } t>0,\\
0&\text{for } t\leq 0,
\end{cases}
\end{equation}
and consider the sub-Riemannian structure on $\R^3$ given by 
the global generating family:
\begin{equation}
X_1=\partial_1,\qquad X_2=\partial_2+x_1\partial_3,\qquad X_3=f(x_1)\partial_3,
\end{equation}
where $x=(x_1,x_2,x_3)$ denote the coordinates in $\R^3$ and $\partial_i$ denotes the derivative with respect to the $i$-th coordinate. The singular region is the set $\mathcal Z=\{ x_1=0\}$. Observe that, on  $\mathcal R_-=\{x_1<0\}$, this is the Heisenberg sub-Riemannian structure on $\R^3$, while, for $\mathcal R_+=\{x_1>0\}$, this is a Riemannian structure. In particular, using the explicit formula  \eqref{eq:popp}, we obtain that the Popp's measure $\mathcal P$ is
\begin{equation}
\mathcal P=\frac{e^{2\theta}}{\sqrt 2}dx_1\wedge dx_2\wedge dx_3,\qquad\text{where}\qquad
\theta(x_1)=\begin{cases}
0&\text{if }x_1<0,\\
\frac{1}{2x_1^2}&\text{otherwise}.
\end{cases}
\end{equation}

Although the above computation shows that the function $\rho$ defined in \eqref{eq:rho} is not a submersion, we can nevertheless compute the effective potential on both sides of the singular region, exploiting the fact that the distance from the singular region is $\delta(x)=|x_1|$. (Here, the reference measure $\varpi$ is taken to be the Lebesgue measure.)

On $\mathcal R_+$ we have $\Veff=(\partial_1\theta)^2+\partial_1^2\theta\sim 1/x_1^6$, which is greater than $3/(4x_1^2)$ in an uniform neighborhood of $\mathcal Z\cap \mathcal R_+$, leading to essential self-adjointness of $\Delta_{\mathcal{P}}$ defined on $C^\infty_c(\mathcal R_+)$. On the other hand, on $\mathcal R_-$ we have $\Veff\equiv 0$, and Theorem \ref{thm:main} does not apply. 
One can check that $\Delta_{\mathcal {P}}$ is not essentially self-adjoint on $C^\infty_c(\mathcal R_-)$ by, e.g., applying a Fourier transform on the $(x_2,x_3)$ variables and analysing the resulting one-dimensional operator.
\end{example}

\section*{Acknowledgments}
\thanks{This research has  been supported by the Grant ANR-15-CE40-0018 of the ANR, by the iCODE institute (research project of the Idex Paris-Saclay). The first author has been partially supported by the GNAMPA Indam project ``Problemi nonlocali e degeneri nello spazio euclideo'' and by ``Fondazione Ing.~Aldo Gini'', Universit\`a degli Studi di Padova. This research benefited from the support of the ``FMJH Program Gaspard Monge in optimization and operation research'' and from the support to this program from EDF. This work has been partially supported by the ANR project ANR-15-IDEX-02. A proceeding version of this paper appeared in \cite{SRQC-actes}, whose last section contains also some remarks on the difficulties arising in presence of tangency points on the singular region.}

\bibliographystyle{abbrv}
\bibliography{ARSelfAdjointness}

\end{document}